\documentclass{siamart1116}

\usepackage{amsfonts}
\usepackage{graphicx}
\usepackage{epstopdf}
\usepackage{algorithmic}
\ifpdf
  \DeclareGraphicsExtensions{.eps,.pdf,.png,.jpg}
\else
  \DeclareGraphicsExtensions{.eps}
\fi

\usepackage{amssymb,amsmath,amscd,amstext,bbm, enumerate, dsfont, mathtools}
\usepackage{array}
\usepackage{multicol}
\usepackage{multirow}
\usepackage[algo2e]{algorithm2e}
\usepackage{algorithm}
\usepackage{subcaption}
\usepackage{graphicx}
\usepackage{hyperref}
\usepackage{tcolorbox}
\usepackage{import}
\usepackage{pifont}
\usepackage{tikz,pgfplots}
\usetikzlibrary{tikzmark,arrows.meta,shapes}
\usetikzlibrary{matrix}
\usepgfplotslibrary{groupplots}
\pgfplotsset{compat=1.14}

\DeclareMathOperator*{\argmax}{argmax}

\newcommand{\prox}{\mathop{\mathrm{prox}}\nolimits}
\newcommand{\Sum}{\sum\limits_{i=1}^M}
\newcommand{\Sumj}{\sum\limits_{j=1}^M}

\newcommand{\modif}[1]{{#1}}
\newcommand{\modifTwo}[1]{{#1}}

\numberwithin{theorem}{section}

\numberwithin{coro}{section}

\newcommand{\TheTitle}{A Distributed Flexible Delay-tolerant Proximal Gradient Algorithm } 
\newcommand{\TheAuthors}{K. Mishchenko, F. Iutzeler, and J. Malick}

\headers{\TheTitle}{\TheAuthors}

\title{{\TheTitle}\thanks{Submitted to the editors \today.}}

\author{
  Konstantin Mishchenko\thanks{KAUST
    (\email{konstantin.mishchenko@kaust.edu.sa}, \url{http://konstmish.github.io/}).}
  \and
  Franck Iutzeler\thanks{Laboratoire Jean Kuntzmann, Univ. Grenoble Alpes (\email{franck.iutzeler@univ-grenoble-alpes.fr}).}
  \and
  J\'er\^{o}me Malick\thanks{CNRS, Laboratoire Jean Kuntzmann  (\email{jerome.malick@univ-grenoble-alpes.fr}).}
}

\usepackage{amsopn}

\ifpdf
\hypersetup{
  pdftitle={\TheTitle},
  pdfauthor={\TheAuthors}
}
\fi

\begin{document}

\maketitle

\begin{abstract}
 We develop and analyze an asynchronous algorithm for distributed convex optimization when the objective can be written as a sum of smooth functions, local to each worker, and a non-smooth function. Unlike many existing methods, our distributed algorithm is adjustable to various levels of communication cost, delays, machines' computational power, and functions' smoothness. A unique feature is that the stepsizes do not depend on communication delays nor number of machines, which is highly desirable for scalability. We prove that the algorithm converges linearly in the strongly convex case, and provide guarantees of convergence for the non-strongly convex case. The obtained rates are the same as the vanilla proximal gradient algorithm over some introduced epoch sequence that subsumes the delays of the system. We provide numerical results on large-scale machine learning problems to demonstrate the merits of the proposed method.
\end{abstract}



  


\section{Introduction}

A broad range of problems arising in machine learning and signal processing can be formulated as minimizing the sum of $M$ smooth functions\;$(f_i)$ and a non-smooth proximable function $g$
\begin{align}
    \label{eq:pb}
    \min_{x\in\mathbb{R}^n}~~ \frac{1}{M}\sum_{i=1}^M f_i(x) 
    ~+~ g(x).
\end{align}
For instance, $f_i$ may represent a local loss and $g$ a non-smooth regularizer that imposes some structure on optimal solutions. Typical examples include the $\ell_1$-regularized regression \cite{tibshirani1996regression} in which $g$ is taken as the $\ell_1$-norm \cite{bach2012optimization}. 

\subsection{Distributed setting}

In this paper\footnote{Our preliminary work in a machine learning context 
\cite{icml2018} presents briefly the asynchronous framework and a theoretical study in the strongly convex case. We extend this work on several aspects with in particular a deeper analysis of the asynchronous setting, the use of local stepsizes, and the study of the general convex case.}, we consider the optimization problem~\eqref{eq:pb} in a distributed setting with $M$ worker machines, where worker $i$ has private information on the smooth function $f_i$. 
More precisely, we assume that each worker $i$ can compute:
\begin{itemize}
    \item the gradient of its local function $\nabla f_i$;
    \item the proximity operator of the common non-smooth function $\prox_g$.
\end{itemize}
We further consider a master slave framework where the workers exchange information with a master machine which has no global information about the problem but only coordinates the computation of agents in order to minimize~\eqref{eq:pb}. 
\modif{Having asynchronous exchanges between the workers and the master is of paramount importance for practical efficiency; indeed, asynchronous algorithms can perform more iterations per second for nearly the same improvement as their synchronous counterparts even with large delays (see e.g.\;the asynchronous parallel fixed point algorithm of\;\cite{hannah2017more}). In the considered setup, as soon as the master receives an update from a worker, it updates the master variable, and sends it to this worker which then carries on its computation.}

This distributed setting covers a variety of scenarios when computation are scattered over distributed devices (computer clusters, mobiles), each having a local part of the data (the locality arising from the prohibitive size of the data, or its privacy\;\cite{shokri2015privacy}), as in federated learning\;\cite{konevcny2016federated}. 
In the large-scale machine learning applications for instance, data points can be split across the $M$ workers, so that each worker $i$ has a local loss function $f_i$ with properties which may be different due to data distribution unevenness. 

\modif{We focus on the setup where (i) the workers' functions differ in their values and the computational complexity of their local oracles (e.g. due to non-i.i.d. unbalanced local datasets in a learning scenario); (ii) the communications between workers and the master are time-consuming (e.g. due to scarce availability or slow communications). This implies that we need to pay a special attention to the delays of workers' updates.}

\subsection{Contributions and outline}

In this distributed setting, we provide an asynchronous algorithm and the associated analysis that adapts to local functions' parameters and can handle any kind of delays. The algorithm is based on totally asynchronous proximal gradient iterations with different stepsizes, which makes it adaptive to the individual functions' properties. In order to subsume delays, we develop a new epoch-based mathematical analysis, encompassing computation times and communication delays, to refocus the theory on algorithmics. \modifTwo{We show convergence in the general convex case and linear convergence in the strongly convex case. More precisely, we show that the proposed method verifies a decrease depending only on the problem parameters (strong convexity and smoothness constants) over meta-iterations, called \emph{epochs}, that subsume the delays between the different workers. This approach thus decouples the convergence between the problem parameters and the delays brought by asynchrony.} The algorithm thus handles the diversity of the previously-discussed applications.

The paper is organized as follows.
In Section~\ref{sec:algo}, we give a description of the algorithm, split into the communication and the optimization scheme, as well as a comparison with the 
most related literature. In Section~\ref{sec:th}, we develop our epoch-based analysis of convergence, separating the general and the strongly convex case.
In Section~\ref{sec:experiments}, we provide illustrative computational experiments on standard $\ell_1$-regularized problems showing the efficiency of the algorithm and its resilience to delays.

\subsection{Related work}\label{sec:existing}

Most existing methods for solving problem~\eqref{eq:pb} in the considered context are based either on parallel stochastic algorithms or on distributed extensions of standard algorithms.

Stochastic algorithms have received a lot of attention, regarding convergence rates, acceleration, parallelization, generalization to non-smooth or sparse gradient cases; see e.g.\;\cite{shalev2013accelerated, johnson2013accelerating, defazio2014saga}.
Parrallel versions of stochastic algorithms have also been proposed where subparts of the data are stored in different machines (Hogwild! \cite{recht2011hogwild}, Distributed SDCA \cite{takavc2015distributed}, Distributed SVRG \cite{lee2015distributed}, ProxASAGA \cite{pedregosa2017breaking}). 
Despite their theoretical properties and practical success in the context of multicore computers, these algorithms are not well-suited for our distributed setting where we focus not only on the number of data accesses, but also on the number of communication steps (see e.g.~\cite{Ma2017arbitrary}). For example, ASAGA \cite{pedregosa2017breaking} makes computations in parallel but does not fit our framework, as it assumes uniform sampling with shared memory between computing parties. Thus, a naive extension of such parallel stochastic methods would be inefficient in practice due to large overheads in communications.

There also exists a rich literature on distributed optimization algorithms with no shared memory. We mention e.g.\;ARock\;\cite{peng2016arock}, Asynchronous ADMM~\cite{zhang2014asynchronous}, COCOA \cite{ma2015adding}, Delayed Proximal Gradient algorithms \cite{aytekin2016analysis,vanli2016stronger}, or dSAGA\;\cite{calauzenes2017distributed}. These methods often have restrictive assumptions about synchrony of communications, or boundedness of the delays between fastest and slowest machines. For instance, the asynchronous distributed ADMM of \cite{zhang2014asynchronous} allows asynchronous updates only until a maximal delay, after which every worker has to wait for the slowest\;one.

Usually, the bounds on delays also impact the stepsizes in algorithms and the convergence rates, as in \cite{peng2016arock,aytekin2016analysis}.
The only other work establishing convergence with unbounded delays is \cite{sun2017asynchronous, hannah2016unbounded} for asynchronous coordinate descent methods (but with decreasing stepsizes).
\modif{\modifTwo{In contrast with all existing literature, we propose a totally asynchronous algorithm that does not require any knowledge about the delays: in practice, delays impact the observed convergence but i) the choice of the stepsize is delay-independent; and ii) our convergence analysis relies on the notion of \emph{epochs} subsuming the delays produced by asynchrony (over these epochs the rate only depends on the problem parameters).}

\modifTwo{Let us finally point out the main improvements over the companion conference paper \cite{icml2018} which presents briefly the  asynchronous framework and a theoretical study in the strongly convex case. In the present paper, we provide a pedagogical study of the mechanisms at play (in Section~\ref{sec:algo}) and a refined analysis covering the non-strongly convex case (in Section~\ref{sec:th} and the appendices).} For a better understanding, we add several illustrative figures explaining the algorithms, the proof techniques, and toy examples. The experiments provided here are complementary to those presented in \cite{icml2018}; in particular, we illustrate the behavior of the algorithm for non-strongly convex objectives.
}

\section{\texttt{DAve-RPG}: Distributed Averaging of Repeated Proximal Gradient}\label{sec:algo}

In this section, we present the proposed \texttt{DAve-RPG} algorithm, where \texttt{DAve} stands for the global communication scheme based on distributed averaging of iterates, and \texttt{RPG} stands for the local optimization scheme, based on repeated proximal-gradient steps. We start by presenting the generic master slave setting and associated notations.

\subsection{Asynchronous Master Slave Framework}
\label{sec:framework}

We consider the master slave model: in order to reach the global objective~\eqref{eq:pb}, the workers exchange information with a master machine. In view of practical efficiency (see e.g. the recent \cite{hannah2017more}), these exchanges are asynchronous: at each moment when the master receives an update from a worker, it revises its master variable and sends it back to the sender.

In compliance with this asynchronous setting, we call iteration/time $k$ (denoted by a superscript $k$), the moment of the $k$-th exchange between a worker and the master, or, equivalently, the $k$-th time the master has updated its master variable.
For a worker $i$ and a time $k$, we denote by $d_i^k$ the delay for $i$ at time $k$, i.e.\;the number of master updates since worker $i$ conversed with the master. More precisely, at time\;$k$, the updating worker $i=i(k)$ suffers no delay (in terms of update in the master variable), i.e.\;$d_i^k=0$, while the delays of the other workers are incremented ($d_j^k = d_j^{k-1}+1 \text{ for all } j\neq i(k)$). 
In addition, we denote by $D_i^k$ the relative delay from the penultimate update, mathematically defined as $D_j^k = d_j^k + d_j^{k-d_j^k-1}+1$ for worker $j$ and time $k$. This asynchronous distributed setup and the corresponding definitions are illustrated in Fig.\;\ref{fig:scheme}.

\begin{figure}
    \centering

\tikzstyle{mybox} = [draw=black , thick, rectangle, rounded corners , inner sep = 10pt]
\tikzstyle{mastertitle} =[fill=red!80!yellow, text=white]
\tikzstyle{slavetitle} =[fill=blue!70!black, text=white]
\tikzstyle{slavetitle2} =[fill=black!50!white, text=white]

\resizebox {\columnwidth} {!} {
\begin{tikzpicture}

\node at (5,3) [mybox] (boxM){%
    \begin{minipage}{70pt}
      $  {\color{red!80!yellow} \overline x^k } = \overline x^{k-1} + {\color{blue!70!black} \Delta}  $
    \end{minipage}
};
\node[mastertitle] at (boxM.north) {\bfseries Master};

\node at (0,0) [mybox] (boxS1){%
    \begin{minipage}{70pt}
    \centering
   \Large
   {\small $(\nabla f_1, \mathrm{prox}_{g})   $}\\
   { \includegraphics[width=10pt]{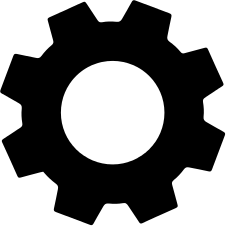}$_1$   }
    \end{minipage}
};
\node[slavetitle2] at (boxS1.north) {\bfseries Worker $1$};

\node at (2.5,0) {\Large ...};

\node at (5,0) [mybox] (boxS2){%
    \begin{minipage}{70pt}
    \centering
   \Large
   {\small $(\nabla f_i, \mathrm{prox}_{g})   $}\\
   {\color{blue!70!black}  \includegraphics[width=10pt]{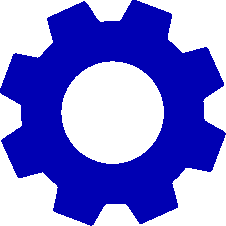}$_i$   $\rightarrow \Delta$ }
    \end{minipage}
};
\node[slavetitle] (ST) at (boxS2.north) {\bfseries Worker $i$};

\node at (7.5,0) {\Large ...};

\node at (10,0) [mybox] (boxS3){%
    \begin{minipage}{70pt}
    \centering
   \Large
   {\small $(\nabla f_M, \mathrm{prox}_{g})   $}\\
   {  \includegraphics[width=10pt]{Figs/gear.png}$_M$   }
    \end{minipage}
};
\node[slavetitle2] at (boxS3.north) {\bfseries Worker $M$};

\draw [->,>=stealth,blue!70!black, thick] (4.8,1.15) to [out=130,in=-130] (4.8,2.35) ;
\node at (4.35,1.75)  {\color{blue!70!black}   $ \Delta$ };

\draw [<-,>=stealth,red!80!yellow, thick] (5.2,1.15) to [out=50,in=-50] (5.2,2.35);
\node at (5.68,1.79)  {\color{red!80!yellow}   $ \overline x^k $ };

\node[scale=0.6] at (5,1.75)  {  $ i=i(k) $ };

\end{tikzpicture}}

\vspace*{0.5cm}

\modifTwo{
\resizebox {\columnwidth} {!} {
     \begin{tikzpicture}
    \draw[thick, ->] (0,0) -- (10,0) node [below] {time};
    \foreach \x in {1,...,19}
    \draw (\x/2, 0.1) -- node[pos=0.5] (point\x) {} (\x/2.0, -0.1);

    \node[scale=1.0,  anchor = east] at (-0.5,0) {updating worker $i=i(k)$};
    
    \node[draw,blue!70!black,ellipse,fill=white,thick, scale = 1.0]  (B) at (7.5,0) {$i$};
    \node[draw,circle,fill=red!80!yellow, scale = 0.8]  at (7.5,0) {};
    \node[scale=0.8] at (7.5,-0.5) {$k = k-d_i^k$};
    \node[draw,blue!70!black,ellipse,fill=white,thick, scale = 1.0] (B) at (5,0) {$i$};
    \node[scale=0.8] at (5,-0.5) {$k-D_i^k$};
    \node[draw,blue!70!black,ellipse,fill=white,thick, scale = 1.0]  (B) at (2.5,0) {$i$};
    \node[draw,blue!70!black,ellipse,fill=white,thick, scale = 1.0]  (B) at (1.5,0) {$i$};
    
    \draw[thick, ->] (0,-1) -- (10,-1) node [below] {time};
    \foreach \x in {1,...,19}
    \draw (\x/2, -1.1) -- node[pos=0.5] (point\x) {} (\x/2.0, -0.9);
    
\node[scale=1.0, anchor = east] at (-0.5,-1) {other worker $j\neq i(k)$};
    
    \node[draw,circle,fill=red!80!yellow, scale = 0.8]  at (7.5,-1) {};
    \node[scale=0.8] at (7.5,-1.5) {$k$};
    \node[draw,white!50!black,ellipse,fill=white,thick, scale = 1.0] (B) at (6.0,-1) {$j$};
    \node[scale=0.8] at (6,-1.5) {$k-d_j^k$};
    \node[draw,white!50!black,ellipse,fill=white,thick, scale = 1.0] (B) at (3.5,-1) {$j$};
    \node[scale=0.8] at (3.5,-1.5) {$k-D_j^k$};
   \node[draw,white!50!black,ellipse,fill=white,thick, scale = 1.0] (B) at (2.0,-1) {$j$};
  \node[draw,white!50!black,ellipse,fill=white,thick, scale = 1.0] (B) at (0.5,-1) {$j$};
    
  \end{tikzpicture}}
  }
  
      \caption{Asynchronous distributed setting and delays notations at iteration $k$.}
    \label{fig:scheme}
  
  \end{figure}

\modifTwo{A key point in this work is that we do \emph{not} assume that the delays are uniformly bounded. We prove instead the convergence of our algorithm and the associated rates using a companion sequence that subsumes the delays.} This places this work in the \emph{totally asynchronous} setting according to Bertsekas and Tsitsiklis' classification \cite[Chap.\;6.1]{bertsekas1989parallel}. Nevertheless, for clarity, we will also provide convergence rates in the partially asynchronous setting with uniform delay boundedness  ($d_i^k\le d$ for all $i,k$) or average delay boundedness ($1/M\sum_{i=1}^M d_i^k \leq (M-1)/2 + d $ for all $k$) in Section~\ref{sec:comp_delay}.

\subsection{\texttt{DAve} Communication scheme}
\label{sec:dave_com}

Our communication scheme is based on maintaining at the master the weighted average of the most updated parameters of the workers. At time $k$, worker $i=i(k)$ finishes the computation of a new \emph{local parameter} $x_i^k$ and the corresponding \emph{adjustment} $\Delta$ corresponding to the weighted difference between its new and former local parameter. As soon as the computation is finished, this adjustment is sent to the master node which, in turn, adds it to its \emph{master parameter} $\overline{x}^k$. The master then immediately sends back this parameter to worker $i$, which can begin a new computation step. 
\modif{During the updates, the master variable is ``locked'' (see
e.g. the description of \;\cite{peng2016arock}), guaranteeing read/write consistency.}

Mathematically, at each time $k$, one has 
\begin{align}\label{eq:prox_average}
    \overline x^k &= \overline x^{k-1} + \Delta \text{  with } \Delta=\pi_i(x_i^k - x_i^{k-D_i^k}) \text{ for } i=i(k)\\
   \text{thus, \quad } \overline x^k &= \Sum \pi_i x_i^{k-d_i^k} =  \Sum  \pi_i \includegraphics[width=10pt]{Figs/gear.png}_i ( \overline x^{k-D_i^k}) \label{eq:prox_average2}
\end{align}
where $\includegraphics[width=10pt]{Figs/gear.png}_i$ represents the computation of worker $i$ (see Figure \ref{fig:scheme}) and the $(\pi_i)_{i=1,..,M}$ are the weights of the workers contributions. \modif{These weights are positive real numbers such that $\sum_{i=1}^M \pi_i = 1$ and are kept fixed over time; their values are derived from  the optimality conditions of~\eqref{eq:pb} and the workers computation. In this paper, the agents will perform (proximal) gradient steps ($\includegraphics[width=10pt]{Figs/gear.png}_i$ = proximal gradient step on $f_i+g$) which leads to the weights given by\;\eqref{eq:pi}.} 

\modif{We see in Eq.~\eqref{eq:prox_average2} that $\overline x^k$ depends on local parameters $(x_i^{k-d_i^k})_i$, which themselves were computed using (once more delayed) global parameters $(\overline x^{k-D_i^k})_i$. A unique feature\footnote{We note that this idea of averaging iterates has also been used in the different context: for variance reduction in incremental methods \cite{defazio2014finito, mokhtari2016surpassing}.} of our distributed algorithm is that at each time, the master variable is as a weighted average of the agents' last contributions. This means that the contribution of each worker in the master variable stays fixed over time even if one worker updates much more frequently than the others. \modifTwo{Though this simple idea might be counterproductive in other contexts, it allows here }the algorithm to cope with heterogeneity in the computing system such as data distribution and agents delays. Roughly speaking, in standard approaches, if an agent has very outdated information, the output of its computation can lead to a counter-productive change, 
generating instability in the algorithm; keeping a fixed average of the contributions offers a counterbalance to such drastic updates. This phenomenon is illustrated in the case where the agents computations are gradients steps in Section~\ref{sec:piag}, notably through Figures~\ref{fig:averaging} and \ref{fig:2d_examples}. }

\subsection{Optimization scheme: Repeated Proximal Gradient \texttt{RPG}}
As the problem features a smooth and a non-smooth part, it is natural that the workers use proximal gradient steps. Furthermore, we allow the repetition of local proximal gradient steps before exchanging with the master, for higher flexibility in the computing time between two exchanges. We present our \texttt{RPG} scheme in 3 stages, explaining the three letters of the name. For more readability, we consider a generic worker $i$ and time $k$ when $i=i(k)$ is the exchanging worker (as represented in Figure\;\ref{fig:scheme}).

$\blacklozenge$ \textbf{G.} If $g\equiv 0$, then each worker may perform a simple gradient step on the last master parameter received $\overline{x}^{k-D_i^k}$:
\begin{equation}
\label{eq:simple_gear}
    x_i^k \leftarrow \overline x^{k - D_i^k} - \gamma_i \nabla f_i(\overline x^{k - D_i^k}),~~~~~~~
    \Delta \leftarrow \pi_i\left(x_i^k - x_i^{k - D_i^k}\right)
\end{equation}
where  $\gamma_i$ is the \emph{local stepsize} at worker $i$ (related only to function $f_i$) and  
\begin{align}
\label{eq:pi}
    \pi_i:= \frac{\frac{1}{\gamma_i}}{\sum_{j=1}^M\frac{1}{\gamma_j}}
\end{align}
is the \emph{proportion} of worker $i$'s contribution, necessary to converge to the correct point.

$\blacklozenge$  \textbf{PG.} For a general non-smooth convex function $g$, we consider the proximity operator, defined for any $\gamma>0$ by
\begin{align*}
    \prox_{\gamma g}(x) = \arg\min_z \left\{ g(z) + \frac{1}{2\gamma} \left\|z - x \right\|^2 \right\}.
\end{align*}
One can extend \eqref{eq:simple_gear} in the same way iteration $\prox_{\gamma g}(x - \gamma\nabla f(x))$ generalizes a gradient step. However, contrary to direct intuition, the proximity operator has to be computed first, leading to a temporary variable $z$, on which is taken the gradient step before exchanging: 
\begin{equation}
\label{eq:advanced_gear}
    z\leftarrow \prox_{\gamma g}(\overline x^{k - D_i^k}),~~~~
    x_i^k \leftarrow z - \gamma_i\nabla f_i(z),
    ~~~~
    \Delta \leftarrow \pi_i\left(x_i^k - x_i^{k - D_i^k}\right)
\end{equation}
with $\gamma$ being the \emph{master stepsize} appearing in \emph{all} proximity operators:
\begin{align}
\label{eq:step}
    \gamma := \frac{M}{\sum_{i=1}^M \frac{1}{\gamma_i}}
\end{align}
equal to the harmonic average of the local stepsizes. Note that our algorithm allows for different local stepsizes, which simplifies parameters tuning as it can be done locally. Then, the proximity operators have to be taken with a separate master stepsize. 


$\blacklozenge$  \textbf{RPG.} Once all computations of iteration \eqref{eq:advanced_gear} are done, the slave could send the adjustment $\Delta$ to the master and get $\overline x^k$ in response. However, the difference between the latest $\overline x^k$ and $\overline x^{k - D_i^k}$ may be small, so the worker would only gain little information from a new exchange. Thus, instead of communicating right away, we suggest to perform additional proximal gradient updates by taking as the starting point  $\overline x^{k - D_i^k} + \Delta$.
The motivation behind this repetition is to lower the burden of communications and to focus on computing good updates.
We will prove later that there is no restriction on the number of repetitions (called $p$ in the algorithm), as any value can be chosen and it can vary freely both across machines and over time.

\tcbset{width=1.0\columnwidth,before=,after=, colframe=black,colback=white, fonttitle=\bfseries, coltitle=white, colbacktitle=red!80!yellow, boxrule=0.2mm}
\begin{algorithm}[H]
\caption*{\texttt{\bfseries DAve-RPG}\label{alg:dave_rpg}}
\centering
\begin{multicols}{2}
\begin{tcolorbox}[title=Master:]
Initialize $\overline x = \overline x^0$, $k=0$\\
\While{not converge}{
    \textbf{when} a worker finishes:\\
    {\color{blue!70!black}Receive adjustment $\Delta$ from it}\\
    $\overline{x} \leftarrow \overline{x} + \Delta$\\
    {\color{red!80!yellow} Send $\overline{x}$ to the agent in return}\\
    $k\leftarrow k+1$
}
Interrupt all slaves\\
\textbf{Output} $ x = \prox_{\gamma g}\left(\overline x\right)$
\end{tcolorbox}

\columnbreak
\tcbset{width=1.03\columnwidth,before=\hspace{-0.2cm}, colframe=black!50!black, colbacktitle=blue!70!black}
\begin{tcolorbox}[title=Slave $i$:]
Initialize $x = x_i = \overline x$, \\
\While{not interrupted by master}{
    {\color{red!80!yellow}Receive the most recent $\overline{x}$}\\
    Select a number of repetitions $p$\\
    $\Delta \leftarrow 0$\\
    \For{$q= 1$ \KwTo p}{
        $z \leftarrow \prox_{\gamma g} (\overline{x} + \Delta)$\\
        $x^+ \leftarrow z - \gamma_i  \nabla f_i(z)$\\
        $\Delta \leftarrow \Delta + \pi_i\left(x^+ - x\right)$\\
        $x \leftarrow x^+$
    }
    {\color{blue!70!black}Send adjustment $\Delta$ to master}
}
\end{tcolorbox}
\end{multicols}
\end{algorithm}

\subsection{Comparison between DAve-(R)PG and PIAG}
\label{sec:piag}

Our algorithm DAve-RPG performs a distributed minimization of the composite problem \eqref{eq:pb} by aggregating the agents contributions. It is closely related to the proximal incremental aggregated gradient (PIAG) method \cite{aytekin2016analysis,vanli2016global}.
We can compare the update of PIAG with the one of $x^k = \prox_{\gamma g}(\overline x^k)$ for DAve-PG\footnote{\modifTwo{For the master, the iteration $k$ reads $x^k = \mathrm{prox}_{\gamma g}(\overline x^k)$ where $\overline x^k$ is the average of the last update of each worker: $\overline x^k = \sum_{i=1}^M \pi_i x_i^k$ (see Eq.~\eqref{eq:prox_average2}). For each worker $i$, $x_i^k$ is the result of the last gradient step performed by this worker on its local function: $x_i^k = x^{k-D_i^k} - \gamma_i \nabla f_i (x^{k-D_i^k}) $ (see Eq.~\eqref{eq:simple_gear}). Putting it all together, we get   $x^k =  \mathrm{prox}_{\gamma g}(  \sum_{i=1}^M  \pi_i x^{k - D_i^k} -   \sum_{i=1}^M   \pi_i \gamma_i\nabla f_i(  x^{k - D_i^k}))$. Finally, this expression can be simplified by noticing that $\pi_i \gamma_i = \gamma/M$ (see Eqs.~\eqref{eq:pi} and \eqref{eq:step}). }} (with one repetition, $p=1$).\\

\noindent\resizebox{\columnwidth}{!}{
\begin{tabular}{c|c}
 DAve-PG     &  PIAG \\
\hline
 & \\
$  x^k =  \prox_{\gamma g}\left(  \Sum  \pi_i x^{k - D_i^k} -  \Sum   \pi_i \gamma_i\nabla f_i(  x^{k - D_i^k})\right)
$
   & $  x^k =   \prox_{\gamma g}\left( x^{k-1} - \gamma \frac{1}{M} \Sum   \nabla f_i( x^{k - D_i^k})\right) $\\
 \hspace*{0.0cm} $   =  \prox_{\gamma g}\left(  \Sum  \pi_i x^{k - D_i^k} -  \gamma\frac{1}{M} \Sum   \nabla f_i( x^{k - D_i^k})\right) $ &
\end{tabular}
}

\vspace*{0.3cm}

These two algorithms are separated by a major difference: PIAG performs an aggregated delayed gradient descent from the most recent main variable $x^{k-1}$ and uses all gradients regardless of corresponding delays. Clearly, if one gradient has not been updated for long time, this update rule may be harmful\modif{, as mentioned at the end of Section~\ref{sec:dave_com}}. On the other hand, DAve-(R)PG performs a similar aggregated delayed gradient descent (with more adaptive local stepsizes) but from the averaged main point $\sum_{i=1}^M  \pi_i \overline x^{k - D_i^k} $. This more conservative update prevents instabilities in the case where some worker is silent for too long, and, thus, is more robust. See Figure~\ref{fig:averaging} for a geometrical illustration.

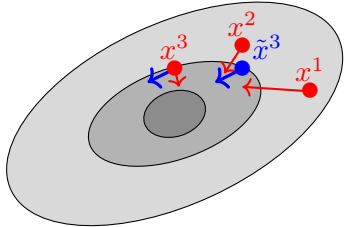
\begin{figure}
    \centering

      \begin{tikzpicture}[scale=0.6]
      
      \coordinate (P1) at (1.5,1.5);
      \coordinate (P2) at (0,1);
      \coordinate (P3) at (3,0.5);
      \coordinate (PM) at ($0.333*(P1) + 0.333*(P2) + 0.333*(P3) $);

      \coordinate (G1) at (-0.4,-0.6);
      \coordinate (G2) at (0.1,-0.4);
      \coordinate (G3) at (-1.5,0.1);
      \coordinate (GM) at ($0.333*(G1) + 0.333*(G2) + 0.333*(G3) $);
       
      \filldraw[draw=black,fill=gray!30,rotate=25] (0,0) ellipse (4 and 2);  
      \filldraw[draw=black,fill=gray!60,rotate=20] (0,0) ellipse (2 and 1);  
      \filldraw[draw=black,fill=gray!90,rotate=18] (0,0) ellipse (0.7 and 0.5);

      \draw[->,thick,color=red] (P1) -- ++(G1);
      \draw[->,thick,color=red] (P2) -- ++(G2);
      \draw[->,thick,color=red] (P3) -- ++(G3);
      \draw[->,very thick,color=blue] (PM) -- ++(GM);
      \draw[->,very thick,color=blue] (P2) -- ++(GM);

      \node[color=red,scale=1.5] at (P1) {\textbullet};
      \node[color=red,scale=1.1,above] at (P1) {$x^2$};
      \node[color=red,scale=1.5] at (P2) {\textbullet};
      \node[color=red,scale=1.1,above] at (P2) {$x^3$};
      \node[color=red,scale=1.5] at (P3) {\textbullet};
      \node[color=red,scale=1.1,above] at (P3) {$x^1$};
      \node[color=blue,scale=1.5] at (PM) {\textbullet};
      \node[color=blue,scale=1.1,above right] at (PM) {$\tilde{x}^3$};

\end{tikzpicture}
    \caption{Let the gray ellipses be the level-sets of a smooth function. In red are represented three iterates $(x^k)_{k=1,2,3}$ and their associated descent directions (taken as the opposite of the gradients computed at these points). The blue dots represent the averaged point $\tilde{x}^3 = (x^1 +  x^2 + x^3)/3 $, while the blue vectors both represent the average of the associated descent directions. We notice that in that situation, descending along the averaged gradient is much more interesting from the averaged point $\tilde{x}^3$ than from the last point $ x^3$. }
    \label{fig:averaging}
\end{figure}

In terms of theoretical results, this conservative approach allows us to get stronger convergence results and better rates as derived in the next section:
\begin{itemize}
    \item the stepsize of PIAG, and, thus, its rate, depends heavily on the maximal delays whereas our stepsize does not depend on any form of delays;
    \item PIAG's stepsize is global and, thus, cannot adapt to each of the workers local functions, while we use locally adapted stepsizes;
    \item no version of PIAG exists with \textit{multiple} proximal gradient steps before exchanging with the master.
\end{itemize}

\smallskip

In terms of performance, before more thorough comparisons, Fig.~\ref{fig:2d_examples} gives an illustration of the benefits of the proposed approach compared to PIAG in terms of iterates behavior. In this plot, we consider two runs of DAve-RPG and PIAG applied to a two dimensional problem where one of the 5 functions/workers takes $10$ times as much time to compute its update as the other workers and consequently produces more delayed updates. The objective used is a sum of 5 quadratics centered around different points and the initial point is (-20, -20) in all cases. Although the stepsize used for PIAG was 10 times smaller (due to its dependence to the delays), the iterates produced by PIAG show chaotic deviations from the optimal point while DAve-RPG steadily converges to the optimum.

\begin{figure}[h!]
\begin{center}
    \includegraphics[scale=0.4]{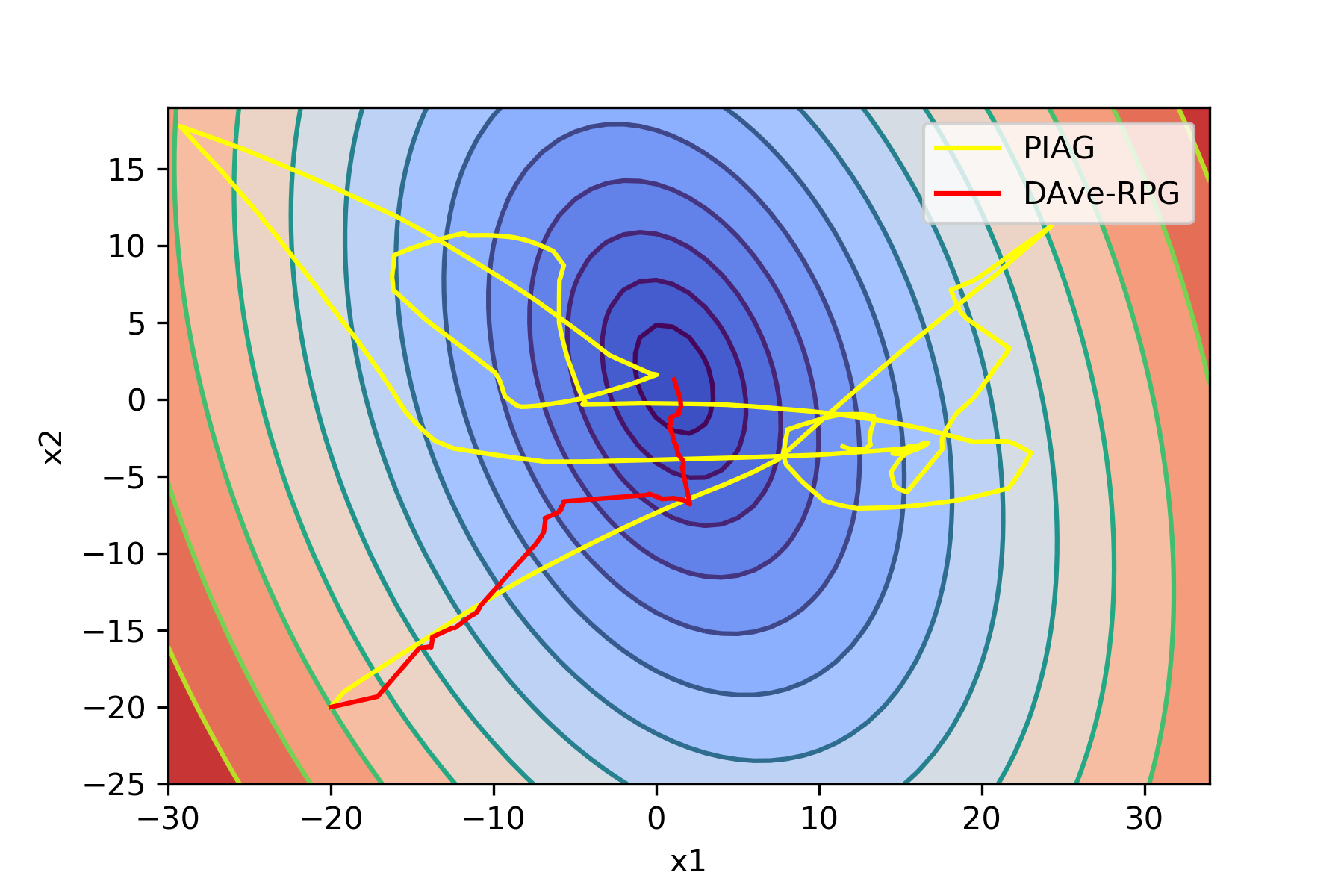}
    \includegraphics[scale=0.4]{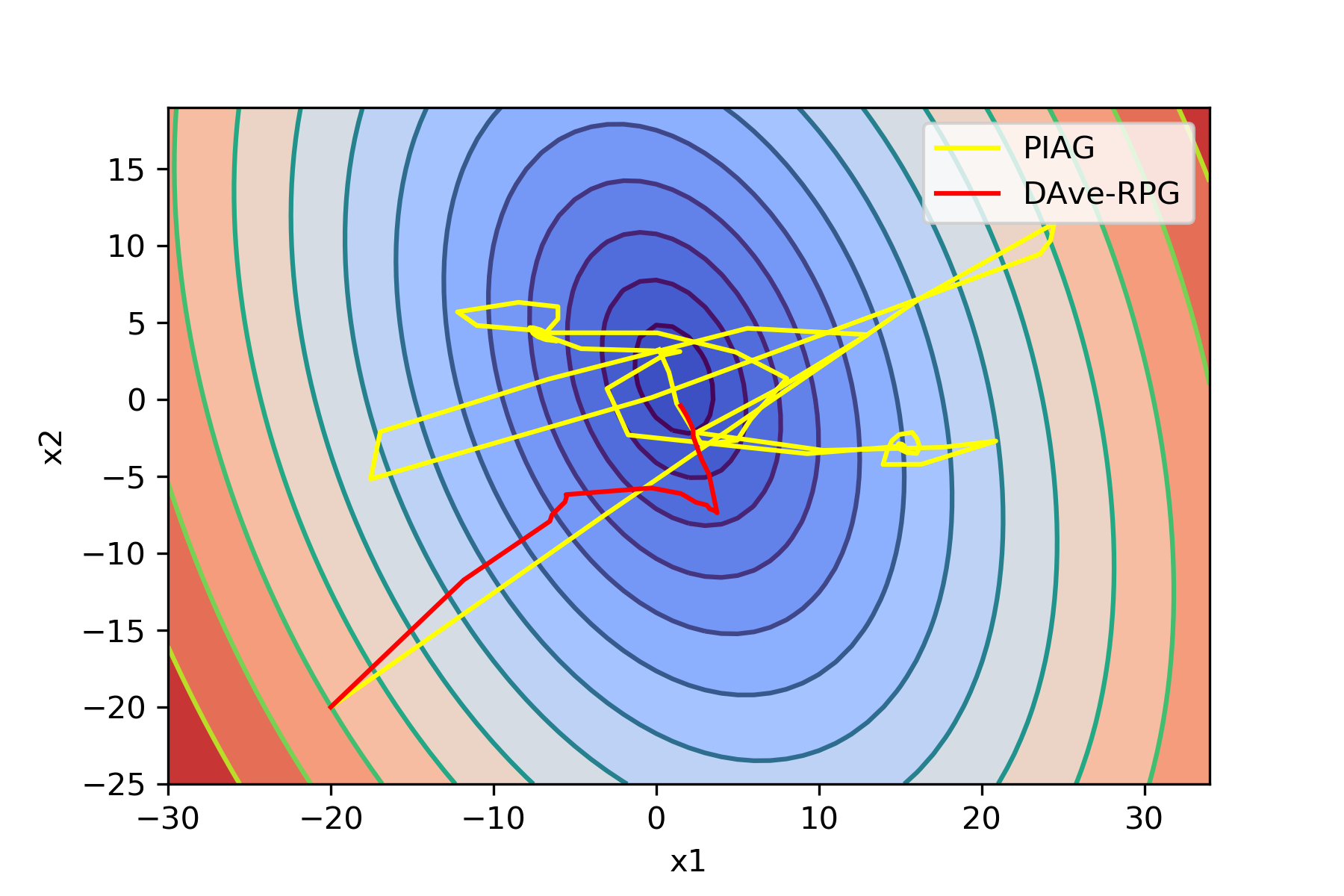}
\end{center}
\caption{Two runs of a two dimensional example with $n=5$ and one worker suffering long delays.}
\label{fig:2d_examples}
\end{figure}

\section{Analysis}
\label{sec:th}

\subsection{Revisiting the clock}\label{sec:epoch}

To the best of our knowledge, all papers on asynchronous distributed methods (except \cite{sun2017asynchronous, hannah2017more, hannah2016unbounded}) assume that delays are uniformly upper bounded by a constant. Moreover, the maximum stepsize is usually highly dependent on this upper bound. In the upcoming results, we show that our algorithm DAve-RPG converges without assuming bounded delays and with the stepsizes depending only on local smoothness and convexity of the functions.

The forthcoming results are based on the careful definition of an \emph{epoch sequence} along which we investigate the improvement of our algorithm (rather than looking at the improvement per iteration).

We define our \emph{epochs sequence} $\{k_m\}_m$ by setting $k_0=0$ and the recursion:
\begin{align*} 
k_{m+1}&=\min\{k:\text{ each machine made at least 2 updates on the interval }[k_m, k]\}\\
&= \min\{k: k-D_i^k \geq k_m \text{ for all } i=1,..,M  \}
\end{align*}

\modif{
\begin{figure}[!ht]
    \centering
    \resizebox {\columnwidth} {!} {
     \begin{tikzpicture}

    \draw[thick] (0,2) -- (7,2);
    \foreach \x in {0,...,13}
    \draw (\x/2, 2.1) -- node[pos=0.5] (point\x) {} (\x/2.0, 1.9);
       
    \node[scale=1.0] at (-0.5,2) {$1$};
    
        \draw[thick] (0,1) -- (7,1) ;
    \foreach \x in {0,...,13}
    \draw (\x/2, 1.1) -- node[pos=0.5] (point\x) {} (\x/2.0,  0.9);
       
    \node[scale=1.0] at (-0.5,1) {$2$};     
     
    \draw[thick] (0,0) -- (7,0) node [below] {};
    \foreach \x in {0,...,13}
    \draw (\x/2, 0.1) -- node[pos=0.5] (point\x) {} (\x/2.0, -0.1);
       
    \node[scale=1.0] at (-0.5,0) {$3$};
    
    
    \node[scale=1.0] at (-1.2,2.5) {Workers $\downarrow$};    
    \node[draw,blue!70!black,ellipse,fill=white,thick, scale = 1.0]  (B) at (6,2.5) {};
    \node[scale=1.0] at (7.2,2.5) {= $1$ update}; 
    
    \node[scale=1.0] at (0,-0.5) {$k_0 = 0$};
    \draw[dashed,thick] (0,-0.3) -- (0,2.3);
    
    \node[scale=1.0] at (2,-0.5) {$k_1$};
    \draw[dashed,thick] (2,-0.3) -- (2,2.3);

    \node[scale=1.0] at (5.5,-0.5) {$k_2$};
    \draw[dashed,thick] (5.5,-0.3) -- (5.5,2.3);

    \node[draw,blue!70!black,ellipse,fill=white,thick, scale = 1.0]  (B) at (0,2) {};
    \node[draw,blue!70!black,ellipse,fill=white,thick, scale = 1.0]  (B) at (0,1) {};
    \node[draw,blue!70!black,ellipse,fill=white,thick, scale = 1.0]  (B) at (0,0) {};

    \node[draw,blue!70!black,ellipse,fill=white,thick, scale = 1.0]  (B) at (0.5,1) {};
    
    \node[draw,blue!70!black,ellipse,fill=white,thick, scale = 1.0]  (B) at (1,1) {};

    \node[draw,blue!70!black,ellipse,fill=white,thick, scale = 1.0]  (B) at (1.5,0) {};

    \node[draw,blue!70!black,ellipse,fill=white,thick, scale = 1.0]  (B) at (2,2) {};
    
    \node[draw,blue!70!black,ellipse,fill=white,thick, scale = 1.0]  (B) at (2.5,1) {};
    
    \node[draw,blue!70!black,ellipse,fill=white,thick, scale = 1.0]  (B) at (3,0) {};
    
    \node[draw,blue!70!black,ellipse,fill=white,thick, scale = 1.0]  (B) at (3.5,2) {};
    
    \node[draw,blue!70!black,ellipse,fill=white,thick, scale = 1.0]  (B) at (4,0) {};

    \node[draw,blue!70!black,ellipse,fill=white,thick, scale = 1.0]  (B) at (4.5,2) {};
    
    \node[draw,blue!70!black,ellipse,fill=white,thick, scale = 1.0]  (B) at (5,2) {};

    \node[draw,blue!70!black,ellipse,fill=white,thick, scale = 1.0]  (B) at (5.5,1) {};

    \node[draw,blue!70!black,ellipse,fill=white,thick, scale = 1.0]  (B) at (6,2) {};
    
    \node[draw,blue!70!black,ellipse,fill=white,thick, scale = 1.0]  (B) at (6.5,0) {};


    \draw[thick,dashed] (7,2) -- (8,2);
    \draw[thick,->] (8,2) -- (12.7,2);
    \foreach \x in {0,...,7}
    \draw (8+\x/2, 2.1) -- node[pos=0.5] (point\x) {} (8+\x/2.0, 1.9);

    \draw[thick,dashed] (7,1) -- (8,1);
    \draw[thick,->] (8,1) -- (12.7,1);
    \foreach \x in {0,...,7}
    \draw (8+\x/2, 1.1) -- node[pos=0.5] (point\x) {} (8+\x/2.0, 0.9);

    \draw[thick,dashed] (7,0) -- (8,0);
    \draw[thick,->] (8,0) -- (12.7,0) node [scale=1.0,below] {iterations};
    \foreach \x in {0,...,7}
    \draw (8+\x/2, 0.1) -- node[pos=0.5] (point\x) {} (8+\x/2.0, -0.1);


    \node[scale=1.0] at (8.5,-0.5) {$k_{m-1}$};
    \draw[dashed,thick] (8.5,-0.3) -- (8.5,2.3);

    \node[scale=1.0] at (11,-0.5) {$k_{m}$};
    \draw[dashed,thick] (11,-0.3) -- (11,2.3);
    
    \node[draw,blue!70!black,ellipse,fill=white,thick, scale = 1.0]  (B) at (8,0) {};

    \node[draw,blue!70!black,ellipse,fill=white,thick, scale = 1.0]  (B) at (8.5,2) {};

    \node[draw,blue!70!black,ellipse,fill=white,thick, scale = 1.0]  (B) at (9,1) {};
    
    \node[draw,blue!70!black,ellipse,fill=white,thick, scale = 1.0]  (B) at (9.5,1) {};
    
    \node[draw,blue!70!black,ellipse,fill=white,thick, scale = 1.0]  (B) at (10,0) {};
    
    \node[draw,blue!70!black,ellipse,fill=white,thick, scale = 1.0]  (B) at (10.5,2) {};
    
    \node[draw,blue!70!black,ellipse,fill=white,thick, scale = 1.0]  (B) at (11,0) {};

    \node[draw,blue!70!black,ellipse,fill=white,thick, scale = 1.0]  (B) at (11.5,2) {};
    
    \node[draw,blue!70!black,ellipse,fill=white,thick, scale = 1.0]  (B) at (5,2) {};
    
  \end{tikzpicture}}
    \caption{Illustration of the epoch sequence for $M=3$ workers. Each circle corresponds to one update, i.e, one iteration. 
    }
    \label{fig:epochs}
\end{figure}
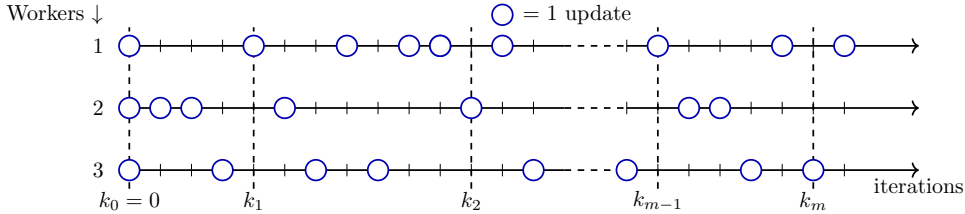

In words, $k_{m}$ is the first moment when all workers have updated twice since $k_{m-1}$. This is illustrated by Figure~\ref{fig:epochs}.
Thus, $k_{m}$ is the first moment when $\overline x^k$ no longer depends directly on information from moments before $k_{m-1}$. Indeed, we have $\overline x^k = \sum_i \pi_i x_i^{k - d_i^k}$ and $x_i^{k-d_i^k}$ was computed using $\overline x^{k - D_i^k}$. 
}

Note that we always have $k_m\geq 2M-1$. Furthermore, in the degenerate case when $M=1$, the epoch sequence corresponds to the time sequence: we have $k_m = m$, because on the interval $[m, m+1]$ there are exactly two updates of the only slave. 

\modif{
In addition, we will assume that the number of epochs goes to infinity, i.e. all workers eventually respond, in order to get convergence.
This is in line with the literature on totally asynchronous algorithms  (see Assumption 1.1 in Chap. 6 of \cite{bertsekas1989parallel}). Nevertheless, our results in the strongly convex case (Theorem~\ref{thm:strong} and its corollary) are still valid even when there is a finite number of epochs; in that case, they mean that the iterates will reach  a ball around the solution of a radius controlled by the number of epochs performed.}

\subsection{Preliminary: local iterations}
\label{sec:local}

To understand why the algorithm converges as a whole, let us first take a close look at how one local iteration of RPG enables iterates to get closer to a local solution. Indeed, a special property of the algorithm is that local variables $(x_i^k)$ do not converge to the same value as the master variable~$\overline x^k$. In contrast, they go to the \emph{local shifted optimal point} $x_i^\star := x^\star - \gamma_i \nabla f_i(x^\star)$.

At worker $i$ and time $k$, $x_i^k = x_i^{k-d_i^k}$ was obtained by $p = p(i,k-d_i^k)$ repetitions of proximal gradient. Starting with the reception of $\overline x^{k-D_i^k}$ and initializing $\Delta^{(0)} = 0$, the $p$ local iterations (indexed by superscripts with parentheses) are obtained by the repetition of
\begin{align*}
z^{(q)} &= \prox_{\gamma g}(\overline x^{k - D_i^k} + \Delta^{(q-1)}),\\
x_i^{(q)} &= z^{(q)}  - \gamma_i\nabla f_i(z^{(q)} )\\
\Delta^{(q)} &= \Delta^{(q-1)} +  \pi_i\left(x_i^{(q)} - x_i^{(q-1)} \right)
\end{align*}
for $q=1,..,p$. Then, $x_i^{k-d_i^k} = x_i^{(p)}$ and $ \Delta^{k-d_i^k} = \Delta^{(p)} $.

The next lemma is fundamental to the analysis of our algorithm. It describes how the local computations go towards their own local shifted optimal point, compared~to
\begin{align}
\label{eq:ak}
\mathbf{a}^k := \max\left(\left\|\overline x^{k} - \overline x^\star\right\|^2, \left\|\overline x^{k}_{-i(k)} - \overline x^\star_{-i(k)}\right\|^2\right) ,   
\end{align}
   where $i(k)$ is the updating agent at time $k$ and 
    \begin{align*}
    \overline x^\star &= \Sum\pi_i x_i^\star,& \overline{x}_{-i}^{k} &= \frac{1}{1 - \pi_i}\sum_{j\neq i} \pi_j x_j^{k},& \overline x^\star_{-i} &=  \frac{1}{1 - \pi_i} \sum_{j\neq i} \pi_j x_j^\star.
    \end{align*}

   \modif{In addition, we have that $x^\star = \prox_{\gamma g}(\overline x^\star)$ by first-order optimality conditions of Problem~\eqref{eq:pb}.}

\begin{lemma}
\label{lem:main}
Let $f_i$ be $\mu_i$-strongly convex ($\mu_i\geq 0$) and $L_i$-smooth, $g$ be convex lsc. Then, with $\gamma_i\in (0, 2/(L_i+\mu_i)]$, we have for any $k$ that after $p_i^{k}$ repetitions
    \begin{align*}
        \left\|x_i^k - x_i^\star\right\|^2 \le \left(1 - \gamma_i \mu_i\right)^2 r_i(p_i^{k})^2 ~ \mathbf{a}^{k-D_i^k}
    \end{align*}
    with $r_i(p) = 1  - \gamma_i \mu_i\sum_{q=1}^{p-1}  (1 - \gamma_i \mu_i)^{q-1}  \pi_i^{q} $.

Furthermore, if $\mu_i = 0$,  with $\gamma_i\in (0, 2/L_i)$, we have for any $k$ and any number of repetitions
        \begin{align*}
        \left\|x_i^k - x_i^\star\right\|^2 \le  \mathbf{a}^{k-D_i^k} - \gamma_i\left(\frac{2}{L_i} - \gamma_i \right)\left\|\nabla f_i(z^{(p)}) - \nabla f_i(x^\star)\right\|^2  
    \end{align*}
    where $z^{(p)}$ is such that $x_i^k = z^{(p)} - \gamma_i \nabla f_i(z^{(p)})$.
\end{lemma}

\begin{proof}
First, as $f_i$ is $\mu_i$-strongly convex and $L_i$ smooth, we have that for any $q=1,..,p$ (see for instance \cite[Chap.~3.4.2]{bubeck-book}),
\begin{align}
&\left\|x^{(q)} - x_i^\star\right \|^2 = \left\|z^{(p)} - \gamma_i \nabla f_i(z^{(q)}) - ( x^\star - \gamma_i \nabla f_i (x^\star) ) \right\|^2 \nonumber\\
\label{eq:one1}    &~\leq \left(1 - \frac{2 \gamma_i \mu_i L_i}{\mu_i + L_i}\right) \left\|z^{(q)} - x^\star\right\|^2 - \gamma_i\left(\frac{2}{\mu_i + L_i} - \gamma_i \right)\left\|\nabla f_i(z^{(q)}) - \nabla f_i(x^\star)\right\|^2 \\
&~\leq \left[ \left(1 - \frac{2 \gamma_i \mu_i L_i}{\mu_i + L_i}\right)  - \mu^2 \gamma_i\left(\frac{2}{\mu_i + L_i} - \gamma_i \right) \right] \left\|z^{(q)} - x^\star\right\|^2 \nonumber \\
\label{eq:one2}       &~= (1 - \gamma_i\mu_i)^2 \left\|z^{(q)} - x^\star\right\|^2.
\end{align}
Then, for $q=1$, we have by non-expansivity of the proximity operator that 
$$\left\|z^{(1)} -x^\star\right\|^2\nonumber \le \left\|\overline{x}^{k-D_i^k} - \overline x^\star\right\|^2 $$
which completes the proof for $p=1$. Going further, for $q\geq 2$, non-expansivity and Jensen's inequality yield
\begin{align}
\left\|z^{(q)} -x^\star\right\|^2\nonumber &\le \left\|\overline{x}^{k-D_i^k} + \Delta^{(q-1)} - \overline x^\star\right\|^2\\
        &=    \left\| \pi_i\left(x^{(q-1)} - x_i^\star\right) + \sum_{j\neq i}\pi_j \left(x_j^{k-D_i^k} - x_j^\star \right) \right\|^2\nonumber\\
        &=   \left\| \pi_i \left(x^{(q-1)} - x_i^\star\right) + (1 - \pi_i)\left(\overline x^{k-D_i^k}_{-i} - \overline x_{-i}^\star\right) \right\|^2\nonumber \\
        &\le        \pi_i \left\| x^{(q-1)} - x_i^\star \right\|^2   + (1-\pi_i) \left\|\overline x^{k-D_i^k}_{-i} - \overline x_{-i}^\star \right\|^2 \nonumber .
\end{align}

Then by induction, using the triangle inequality instead of convexity, one gets that for $p\geq 2$
(and using $\beta_i=(1 - \gamma_i \mu_i) \pi_i$)
\begin{align}
\left\|z^{(p)} - x^\star\right \| 
&\le \pi_i\left\|x^{(p-1)}-x_i^\star\right\| + (1 - \pi_i)\left\|\overline x^{k-D_i^k}_{-i} - \overline x_{-i}^\star\right\|\nonumber \\
&\le \beta_i\left\|z^{(p-1)}-x^\star\right\| + (1 - \pi_i) \sqrt{\mathbf{a}^{k-D_i^k}} \nonumber \\
&\le \beta_i^{p-1} \left\|z^{(1)}-x^\star\right\|  +  \left[ \sum_{q=1}^{p-1}  \beta_i^{q-1} (1 - \pi_i) \right] \sqrt{\mathbf{a}^{k-D_i^k}} \nonumber \\
&\le \beta_i^{p-1} \left\|\overline x^{k - D_i^k}-x^\star\right\|  +  \left[ \sum_{q=1}^{p-1}  \beta_i^{q-1} (1 - \pi_i) \right] \sqrt{\mathbf{a}^{k-D_i^k}} \label{eq:recursion_for_limited_p_and_d} \\
&\le \beta_i^{p-1} \sqrt{\mathbf{a}^{k-D_i^k}}   +  \left[ \sum_{q=1}^{p-1}  \beta_i^{q-1} (1 - \pi_i) \right] \sqrt{\mathbf{a}^{k-D_i^k}} \nonumber \\
&=   \left[ \beta_i^{p-1}  + \sum_{q=0}^{p-2}\beta_i^q  - \frac{1}{1-\gamma_i\mu_i}\sum_{q=1}^{p-1}\beta_i^q  \right] \sqrt{\mathbf{a}^{k-D_i^k}} \nonumber \\
 &=   \underbrace{ \left[ 1  - \frac{\gamma_i \mu_i}{1-\gamma_i \mu_i} \sum_{q=1}^{p-1} \beta_i^q  \right] }_{=  r_i(p) } \sqrt{\mathbf{a}^{k-D_i^k}}  \nonumber 
\end{align}noting that $i=i(k-D_i^k)$ was updating at time $k-D_i^k$ by definition.
Using the last inequality on top of\;\eqref{eq:one1} or \eqref{eq:one2} leads to the claim, noting that $r_i(p) = 1$ for all $p$ when $\mu_i=0$.
\end{proof}

\subsection{Convergence results}

\modif{
In this section, we analyze the convergence of our algorithm, first in the strongly convex case, and second in the general case. In both cases, our results allow us to choose the same stepsize as for vanilla gradient descent (without any dependence on the delays). The derived rates involve the \emph{number of epochs} rather than the number of iterations. In Section~\ref{sec:comp_delay}, we examine how these rates translate in terms of number of iteration under boundedness of the delays in order to compare with the literature. 
}

\subsubsection{Linear convergence in the strongly convex case}
If all the local functions $(f_i)$ are strongly convex, the convergence of our algorithm is linear on the epoch sequence. 

\begin{theorem}[Strongly convex case]
\label{thm:strong}
Let the functions $(f_i)$ be $\mu_i$-strongly convex ($\mu_i>0$) and $L_i$-smooth. Let $g$ be convex lsc. Using $\gamma_i \in (0, \frac{2}{\mu_i + L_i}]$, DAve-RPG converges linearly on the epoch sequence $(k_m)$, with the rate $\rho := \min_i\gamma_i \mu_i$. More precisely, for all $k\in [k_m, k_{m+1})$
\begin{align*}
    \left\| x^k - x^\star \right\|^2 \le \left(1 - \rho  \right)^{2m} \max_i\left\|x_i^0-x_i^\star\right\|^2,
\end{align*}
with the shifted local solutions $x_i^\star = x^\star - \gamma_i\nabla f_i(x^\star)$.
\end{theorem}
\begin{proof}
First, for any $i$ and any $k\in[k_m, k_{m+1})$, we have from Lemma~\ref{lem:main} 
\begin{align*}
    \left\|x_i^k - x_i^\star\right\|^2 &= (1-\gamma_i\mu_i)^2 r_i(p_i^{k})^2 ~~ \mathbf{a}_i^{k-D_i^k} \leq (1-\rho)^2 ~~ \mathbf{a}_i^{k-D_i^k}
\end{align*}

Thus, for any  $k\in[k_m;k_{m+1})$,
\begin{align}
\label{eq:bound_by_max_j}
    \left\|\overline x^k - \overline x^\star \right\|^2 &\le \Sum\pi_i \|x_i^{k} - x_i^\star \|^2 =  \Sum\pi_i \|x_i^{k-d_i^k} - x_i^\star \|^2\nonumber\\
    &\le (1-\rho)^2 ~~   \Sum\pi_i \mathbf{a}^{k-D_i^k}  \le (1-\rho)^2 ~~   \max_i \mathbf{a}^{k-D_i^k}
\end{align}
Similarly, for any $j$
\begin{align}
\label{eq:bound_by_max_j_2}
    \left\|\overline{x}^{k}_{-j} - \overline x^\star_{-j} \right\|^2
    &\le (1 - \pi_j)^{-1}\sum_{i\neq j} \pi_i\left\|x_i^{k - d_i^k} - x_i^\star\right\|^2 \le (1-\rho)^2 ~~   \max_i \mathbf{a}^{k-D_i^k} .
\end{align}
Finally, we get 
\begin{align*}
    \mathbf{a}^{k} \le \left(1 - \rho \right)^2 \max_i \mathbf{a}^{k-D_i^k}
\end{align*}
which is the workhorse for the rest of the proof.

Let $m>0$ and $k\in[k_m, k_{m+1})$, then the definition of the epoch sequence $(k_m)$ gives $k-D_i^k \ge k_{m-1}$ and then
\begin{align*}
    \mathbf{a}^{k} \le \left(1 - \rho \right)^2 \max_i \mathbf{a}^{k-D_i^k} \le \left(1 - \rho \right)^2 \max_{k'\in[k_{m-1},k)} \mathbf{a}^{k'}
\end{align*}
and applying this inequality sequentially to $k_m, k_m+1,\dotsc, k_{m+1} - 1$, we get
    \begin{align}
      \label{eq:km}  
      \mathbf{a}^{k_m} &\le (1 - \rho)^2\max_{k'\in [k_{m-1}, k_m)}  \mathbf{a}^{k'},\\
   \nonumber    
     \mathbf{a}^{k_m+1} &\le  (1 - \rho)^2 \max\left( \max_{k'\in [k_{m-1}, k_m)}  \mathbf{a}^{k'} , \mathbf{a}^{k_m} \right) \\
      \nonumber       
      &\le  (1 - \rho)^2 \max_{k'\in [k_{m-1}, k_m)}  \mathbf{a}^{k'}  ~~~~~ \text{(using Eq.~\eqref{eq:km})} \\
     \nonumber        &~~~~~~~~~ ... \\
       \max_{k\in [k_m, k_{m+1})}  \mathbf{a}^{k} &\le (1 - \rho)^2\max_{k'\in [k_{m-1}, k_m)} \mathbf{a}^{k'}\nonumber \\
        &\le (1 - \rho)^{2m} \max_{k'< k_0}  \mathbf{a}^{k'} \le (1 - \rho)^{2m} \max_i\left\|x_i^0 - x_i^\star\right\|^2 \nonumber.
    \end{align}
Finally, since the proximity operator of a convex function is non-expansive, we have for all  $k\in[k_m;k_{m+1})$, 
\begin{align*}
     \| x^k - x^\star \|^2 &= \|\prox_{\gamma g} (\overline x^k ) - \prox_{\gamma g} (\overline x^\star ) \|^2 \le \|\overline x^k - \overline x^\star\|^2 \\
     &\le    \max_{k\in [k_m, k_{m+1})}  \mathbf{a}^{k} \le (1 - \rho)^{2m} \max_i\left\|x_i^0 - x_i^\star\right\|^2
\end{align*}
which concludes the proof.
\end{proof}

Notice that the rate provided by this theorem is valid for any choice of number of local iterations at any worker/time. The local contraction at agent $i$ can indeed be improved by doing $p$ local repetitions by a factor  
$$ r_i(p) =   1  -  \gamma_i \mu_i\sum_{q=1}^{p-1}  (1 - \gamma_i \mu_i)^{q-1}  \pi_i^{q} =  1  -  \gamma_i \mu_i \pi_i \frac{1 - (1 - \gamma_i \mu_i)^{p-1}  \pi_i^{p-1} }{1 - (1 - \gamma_i \mu_i)  \pi_i }   $$
where $r_i(1) = 1$ and $r_i$ is decreasing with $p$ and lower-bounded by 
$$ r_i(\infty) =    1  -  \frac{ \gamma_i \mu_i \pi_i}{1  - (1 - \gamma_i \mu_i) \pi_i}.$$

If all workers, or at least the ones with the slowest rates, perform several local iterations, the rate can thus be improved as stated by the following result. However, local iterations practically slow down the actual time between two epochs thus the number of local repetitions have to be carefully tuned in practice. The flexibility allowed by our algorithm enables a wide range of selection strategies such as online tuning, stopping the local iterations after some fixed time, etc.

\begin{corollary}[Tighter rates for the strongly convex case]
\label{cor:strong}
Let the functions $(f_i)$ be $\mu_i$-strongly convex ($\mu_i>0$) and $L_i$-smooth. Let $g$ be convex lsc. Using $\gamma_i \in (0, \frac{2}{\mu_i + L_i}]$, DAve-RPG converges linearly on the epoch sequence $(k_m)$, in the sense that for all $k\in [k_m, k_{m+1})$
\begin{align*}
    \left\| x^k - x^\star \right\|^2 \le \left( \Pi_{\ell=1}^m  \alpha_\ell \right) \max_i\left\|x_i^0-x_i^\star\right\|^2,
\end{align*}
with $\alpha_\ell = \max_{i,k\in[k_\ell, k _{\ell+1})} (1-\gamma_i\mu_i)^2r_i(p_i^k)^2$ and  $x_i^\star = x^\star - \gamma_i\nabla f_i(x^\star)$.

In particular, the rate can be uniformly improved to $\alpha = \max_{i,k} (1-\gamma_i\mu_i)^2r_i(p_i^k)^2$.
\end{corollary}

\subsubsection{Convergence and sublinear rate in the general case}\label{sec:cv_gen}

When Problem\;\eqref{eq:pb} is not strongly convex, iterates convergence still holds with the fixed usual stepsizes at the expense of a sublinear rate.

\begin{theorem}[Convergence in the general case]
\label{general_conv_thm}
Let $(f_i)$ be convex $L_i$-smooth, $g$ be convex lsc, and $\gamma_i\in(0,2/L_i)$. Then, if $x^\star$ is the unique minimizer of~\eqref{eq:pb}, the sequence $(x^k)$ converges to $x^\star$. Moreover, if Problem\;\eqref{eq:pb} has multiples minimizers, then $(x^k)$ still converges to a minimizer of \eqref{eq:pb}, under two additional assumptions: (i) the difference between two consecutive epochs $k_{m}-k_{m-1}$ is uniformly bounded, (ii) the number of inner loops is uniformly bounded.
\end{theorem}

\modif{ 
From a mathematical point of view, this result and its proof are the main technical novelties of this paper. We put below the proof of the first part of the result: convergence under no additional assumptions when \eqref{eq:pb} has a unique minimizer. For readability, we postpone to Appendix~\ref{apx:gen} the proof the second part when \eqref{eq:pb} has multiple minimizers. Note that this second part requires an assumption on bounded delays (see more in the discussion of Section \ref{sec:comp_delay}) but no knowledge about this bound (which does not appears in the stepsize range or in the proof).}

\modif{
\begin{proof}
For any $i$ and any $k\in[k_m;k_{m+1})$, we have from Lemma~\ref{lem:main} 
    \begin{align}
    \label{eq:f}
        \left\|x_i^k - x_i^\star\right\|^2 \le  \mathbf{a}^{k-D_i^k} - \gamma_i\left(\frac{2}{L_i} - \gamma_i \right)\left\|\nabla f_i(z^{(p)}) - \nabla f_i(x^\star)\right\|^2  
    \end{align}
    where $\mathbf{a}^{k-D_i^k}$ is the error at time $k-D_i^k$ (see \eqref{eq:ak}) and $z^{(p)}$ is such that $x_i^k = z^{(p)} - \gamma_i \nabla f_i(z^{(p)})$.
Thus, as in Theorem~\ref{thm:strong}, for any  $k\in[k_m;k_{m+1})$, we have by dropping the last term
\begin{align*}
    \left\|\overline x^k - \overline x^\star \right\|^2 &\le \Sum\pi_i \|x_i^{k} - x_i^\star \|^2 =  \Sum\pi_i \|x_i^{k-d_i^k} - x_i^\star \|^2\le   \Sum\pi_i \mathbf{a}^{k-D_i^k}  \le    \max_i \mathbf{a}^{k-D_i^k}.  \nonumber
\end{align*}
Similarly, for any $j$
\begin{align*}
    \left\|\overline{x}^{k}_{-j} - \overline x^\star_{-j} \right\|^2
    &\le (1 - \pi_j)^{-1}\sum_{i\neq j} \pi_i\left\|x_i^{k - d_i^k} - x_i^\star\right\|^2 \le \max_i \mathbf{a}^{k-D_i^k}  \nonumber.
\end{align*}
Finally, we get $ \mathbf{a}^{k} \le \max_i \mathbf{a}^{k-D_i^k} $ from which we can prove using the same arguments as in the proof of Theorem~\ref{thm:strong}
\begin{align*}
          \max_{k\in [k_m, k_{m+1})}  \mathbf{a}^{k} &\le \max_{k'\in [k_{m-1}, k_m)} \mathbf{a}^{k'}, 
\end{align*}
which means that 
\begin{align}
\label{eq:bk}
 \mathbf{b}^m :=\max_{k\in [k_m, k_{m+1})}  \mathbf{a}^{k}
\end{align}
is non-increasing, so that it converges to a non-negative value $\mathbf{b}$.
Getting back to~\eqref{eq:f}, we get that for any $i$ and any $k\in[k_m;k_{m+1})$,
\begin{align*}
     \left\|x_i^k - x_i^\star\right\|^2 \le  \mathbf{a}^{k-D_i^k} \le  \mathbf{b}^{m-1}
\end{align*}
thus when $m\to\infty$, we get that 
\begin{align}\label{eq:b}
\limsup_k  \left\|x_i^k - x_i^\star\right\|^2  \le \mathbf{b}.
\end{align}
The remainder of the proof consists in proving that $\mathbf{b}=0$.

Let $(l^m)$ be a time sequence realizing the max in~\eqref{eq:bk}, i.e.
\begin{align}
\label{eq:lm}
  l^m \in \argmax_{k\in [k_m, k_{m+1})}  \mathbf{a}^{k}
\end{align}
then, be get that $ \mathbf{a}^{l^m} \to \mathbf{b}$ as $m\to\infty$.
We have now two cases: 1) when $\mathbf{a}^{l^m} = \|\overline x^{l^m}_{-i(l^m)} - \overline x^\star_{-i(l^m)}\|^2$ infinitely often; and 2) when  $\mathbf{a}^{l^m} = \|\overline x^{l^m} - \overline x^\star \|^2$ infinitely often. 

We can show that the first case is impossible. In order to ease the reading, we report the proof of this statement at the end of the proof. So we consider now that the sequence
\begin{align*}
  l_1^m &\in \arg\max_{k\in [k_m, k_{m+1})}  \mathbf{a}^{k} \text{ if  }  \mathbf{a}^{l^m}  = \left\|\overline x^{l^m} - \overline x^\star\right\|^2 ~~~\text{ and }~~~ l_1^m = l_1^{m-1} \text{ otherwise}
\end{align*}
and we have that $l_1^m \to \infty$ when $m\to \infty$. We can extract a subsequence $(s^m)$ of $(l_1^m)$ such that $(\overline x^{s^m} , x_1^{s^m}, .. , x_M^{s^m} )$ converges to  $(\overline x , x_1, .. , x_M )$ with $\overline x = \sum_{i=1}^M \pi_i x_i$. 
We are going to show that these points are the limits of all the sequences. Later, the associated $ x := \prox_{\gamma g}(\overline x)$ will also come into play.

We first observe that 
 \begin{align*}
  \mathbf{b} = \lim_{m\to \infty}   \mathbf{a}^{l_1^m}  = \lim_{m\to \infty}   \left\|\overline x^{l_1^m} - \overline x^\star\right\|^2 \leq \sum_{i=1}^M \pi_i   \limsup_{k\to \infty} \left\|x_i^{k} -  x_i^\star\right\|^2 \leq \mathbf{b}.
 \end{align*}
This tells us that
\begin{align*}
   \mathbf{b} =  \left\|\overline x  - \overline x^\star\right\|^2  =   \left\| \sum_{i=1}^M \pi_i    (x_i -  x_i^\star) \right\|^2  \leq \sum_{i=1}^M \pi_i    \left\|x_i -  x_i^\star\right\|^2 \leq  \mathbf{b} 
\end{align*}
and this inequality can only be satisfied if  $ x_i - x_i^\star = x_j - x_j^\star $ for any $i,j$ by direct computation (see e.g. \cite[Lemma 2.13]{bauschke2011convex}). Thus,
\begin{align}
\label{eq:lim_of_xi}
    x_i-x_i^\star=\sum_{j=1}^M\pi_j (x_j-x_j^\star)=\overline{x}-\overline x^\star
\end{align}
which leads to
\begin{align}
\label{eq:x_minus_xi_convergence}
    \overline{x}^{s^m} - x_i^{s^m} \rightarrow \overline x - x_i = \overline x^\star-x_i^\star.
\end{align}

We turn now our attention to convergence of gradients at times ${s^m}$. Rearranging~\eqref{eq:f} and taking the limit, we get first
\begin{align*}
    \limsup_m \left\|\nabla f_i(z_i^{s^m}) - \nabla f_i(x^\star) \right\|^2
    &\le \frac{1}{\omega_i}\left(\lim_m \mathbf{b}^m - \lim_m \left\|x_i^{s^m} - x_i^\star\right\|^2\right) = 0
\end{align*}
 where $z_i^{s^m}$ is such that $x_i^{s^m} = z_i^{s^m} - \gamma_i \nabla f_i(z_i^{s^m})$ and $ \omega_i =  \gamma_i (  2/L_i  - \gamma_i )$. Thus, 
 \begin{align}
     \label{eq:gradc}
     \nabla f_i(z_i^{s^m}) \to \nabla f_i(x^\star)
 \end{align}
 and by definition we get 
 \begin{align*}
     z_i^{s^m} = x_i^{s^m} + \gamma_i\nabla f_i (z_i^{s^m}) \to  x_i + \gamma_i \nabla f_i(x^\star).
 \end{align*}
Define for each $i$ and $k$ vector $w_i^k$ as the one used to get $z_i^k$, i.e.\;$ z_i^k = \prox_{\gamma g}\left(w_i^k\right)$. Using the firm non-expansiveness of the proximal operator $g$ (see \cite[Lemma 12.27]{bauschke2011convex}), we obtain
\begin{align*}
    \left\|x_i^k - x_i^\star\right\|^2      &\le \left\|z_i^k - x^\star\right\|^2 = \left\|\prox_{\gamma g}(w_i^k) - \prox_{\gamma g}(\overline x^\star\right)\|^2 \nonumber \\
    &\le \left\|w_i^k - \overline x^\star\right\|^2 - \left\|z_i^k  - w_i^k - \left(x^\star - \overline x^\star\right)\right\|^2 \nonumber \\
    &\le \mathbf{a}^{k-D_i^k} - \left\|z_i^k  - w_i^k - \left(x^\star - \overline x^\star\right)\right\|^2,  \label{eq:z_i}
\end{align*}
which yields 
\begin{align}
\limsup_m & \left\|z_i^{s^m}  - w_i^{s^m} - \left(x^\star - \overline x^\star\right)\right\|^2     \le \lim_m \mathbf{b}^{m} - \lim_m \left\|x_i^{s^m} - x_i^\star\right\|^2 = 0. \label{eq:zw}
\end{align}
This yields in turn, by \eqref{eq:lim_of_xi}, as $ w_i^{s^m}  - z_i^{s^m} + \left(x^\star - \overline x^\star\right) \rightarrow 0$, that
\begin{align*}
    w_i^{s^m} &\rightarrow x_i + \gamma_i \nabla f_i(x^\star) -\left(x^\star - \overline x^\star\right) = (x_i^\star + \overline x - \overline x^\star) + \gamma_i\nabla f_i(x^\star) - \left( x^\star - \overline x^\star \right)\nonumber = \overline x.
\end{align*}

To finish the proof, we consider the point $ x = \prox_{\gamma g}(\overline x)$, and we observe that the non-expansiveness  of $\prox_{\gamma g}$ gives
\begin{align*}
    \left\|  x^{s^m} -  x\right\|&= \left\|\prox_{\gamma g}(\overline x^{s^m}) - \prox_{\gamma g}(\overline x)\right\| \le \left\|\overline x^{s^m} - \overline x\right\|\rightarrow 0 \\
    \text{and } ~~ \left\|  z_i^{s^m} -  x\right\|&= \left\|\prox_{\gamma g}(\overline w_i^{s^m}) - \prox_{\gamma g}(\overline x)\right\| \le \left\|\overline w_i^{s^m} - \overline x\right\|\rightarrow 0 \text{ for any } i.
\end{align*}
Therefore, the $L_i$-Lipschitz continuity of $\nabla f_i$ gives for any $i$,
\begin{align*}
    \left\| \nabla f_i(z_i^{s^m}) -  \nabla f_i( x)  \right\|&\le L_i\left\| z_i^{s^m} - x\right\|\rightarrow 0
\end{align*}
so $\nabla f_i( x) = \nabla f_i(x^\star)$ using~\eqref{eq:gradc}.
Finally, \eqref{eq:zw} also gives us that $  \lim_m \|z_i^{s^m}  - w_i^{s^m} - (x^\star - \overline x^\star ) \|^2  =    \| x - \overline x -  (x^\star - \overline x^\star ) \|^2  = 0$ so $x - \overline x  = x^\star - \overline x^\star $.

Thus, for any $i$, we get from the definitions of $x^\star, \overline x^\star, x_i^\star = x^\star - \gamma_i \nabla f_i(x^\star) $ and the characterization $x = \prox_{\gamma g }(\overline x) \Leftrightarrow x + \gamma \partial g(x) \ni \overline{x}$ that
\begin{align*}
    \gamma \nabla f( x) &= \sum_{i=1}^M \pi_i \gamma_i \nabla f_i( x) =   \sum_{i=1}^M \pi_i \gamma_i \nabla f_i( x^\star ) =  \sum_{i=1}^M \pi_i (  x^\star - x_i^\star ) \\
    &= x^\star - \overline{x}^\star = x - \overline{x} \in \gamma \partial g(x) 
\end{align*}
thus $ 0 \in \partial (f+g)(x) $.
We can conclude by using the unique minimizer assumption on $f+g$: we get indeed that $x = x^\star$, so $\overline x = \overline x^\star$. This leads to 
\begin{align*}
    \mathbf{b} = \lim_m \left\|\overline x^{s^m} - \overline x^\star\right\|^2  = 0
\end{align*}
which directly implies that $x^k \to x^\star$, and ends the proof. Note that we use the fact that we are in the case of unique minimizer only here for the final conclusion.
\bigskip

\noindent
\emph{Proof of the statement that  $l_1^m \not\to \infty$ when $m\to \infty$ is impossible.}   

\noindent
In this case, we have $ \limsup_{m\to \infty}   \|\overline x^{l^m} - \overline x^\star \|^2  < \mathbf{b}$. Introducing 
\[
l_2^m \in \arg\max_{k\in [k_m, k_{m+1})}  \mathbf{a}^{k} \text{ if  }  \mathbf{a}^{l^m}  = \left\|\overline x^{l^m}_{-i(l^m)} - \overline x^\star_{-i(l^m)}\right\|^2 ~~~\text{ and }~~~ l_2^m = l_2^{m-1} \text{ otherwise}
\]
we have that $l_2^m\to\infty$. We also have
  \begin{align*}
  \mathbf{b} = \lim_{m\to \infty}   \mathbf{a}^{l_2^m}  = \lim_{m\to \infty}   \left\|\overline x_{-i(l_2^m)}^{l_2^m} - \overline x_{-i(l_2^m)}^\star\right\|^2 \leq  \mathbf{b};
 \end{align*}
and we are going to show that it leads to a contradiction.

We extract a subsequence $(s_2^m)$ from $(l_2^m)$ such that $i(s_2^m) = i$ is fixed and and $(\overline x^{s_2^m} , x_1^{s_2^m}, .. , x_M^{s_2^m} )$ converge to  $(\overline x , x_1, .. , x_M )$. Using $ \|\overline x_{-i}^{s_2^m} - \overline x_{-i}^\star \|^2 \to  \mathbf{b} $, one can repeat the arguments of the other case to prove that for any $j, \ell\neq i$
\begin{equation}\label{eq:i}
x_j - x_j^\star =  x_\ell  - x_\ell^\star.
\end{equation}
We would like to have this property for another $i'\neq i$, so that we would have equality of all the $x_j - x_j^\star$ which would yield
\begin{align*}
\label{eq:contradiction}
    \left\|\overline x - \overline x^\star\right\|^2 = \left\|\Sumj\pi_j \left( x_j - x_j^\star\right)\right\|^2 = \Sumj\pi_j \left\| x_j - x_j^\star\right\|^2 = \mathbf{b}
\end{align*}
and then contradicts $ \limsup_{m\to \infty}   \|\overline x^{l^m} - \overline x^\star \|^2  < \mathbf{b}$. 

We have left to prove the existence of this second machine $i'$ with the same property \eqref{eq:i}.
If the machine $i$ is the only machine that is making updates infinitely many times on times $s_2^m$, we have that for any $j\neq i$, $\|x_j^{s_2^m} - x_j^\star \|^2 \to \mathbf{b}$. From Lemma~\ref{lem:main}, it follows that $  \|\overline x_{-j}^{s_2^m-D_j^{s_2^m}} - \overline x_{-j}^\star \|^2 \to \mathbf{b}$ so we can unite the two sequences $(s_2^m)$ and $(s_2^m-D_j^{s_2^m})$ to get a new sequence with the same properties but two slaves making updates infinitely many times. Without loss of generality, we then have that at least workers $i$ and $i'$ and then we get \eqref{eq:contradiction}, and the contradiction follows.
\end{proof}}

Besides convergence, we can also establish the rate of our algorithm in the general case, showing that it matches the one of vanilla gradient descent along the epoch sequence. The proof of this result is reported in Appendix~\ref{apx:rate}.

\begin{theorem}[Rate of convergence]
\label{gradient_rates_theorem}
Let the functions $(f_i)$ be convex $L_i$-smooth and $g$ be convex lsc. Then, for  $\gamma_i\in(0,2/L_i)$ and any $k\in [k_m, k_{m+1})$
    \begin{align*}
        \min_{\substack {k'\le k}}~~\|\partial F(x^{k'})\| \le \frac{2\sqrt{2}}{\sqrt m}\frac{\max_i \left\|x_i^0 - x_i^\star\right\|}{\min_j\left(\gamma_j\sqrt{2 - \gamma_j L_j} \right)},
    \end{align*}
    where $\|\partial F(x^{k'})\|\coloneqq \min_{h\in \partial F(x^{k'})}\|h\|$.  
\end{theorem}

\subsection{Comparison of the results with the literature}
\label{sec:comp_delay}

The main feature of the epoch sequence introduced in Section\;\ref{sec:epoch} is that it automatically adapts to variations of behaviors of machines across time (such as one worker being slow at first that gets faster with time). The sequence then allows for a intrinsic convergence analysis without any knowledge of the delays, as shown in the previous sections. This simple but powerful remark is one of the main technical contributions of this paper.
\modif{
For comparisons with the literature, the following result provides explicit connections between number of iterations and number of epochs with two standard bounds on delays uniformly in time\footnote{A notable exception allowing for potentially unbounded delays is the preprint\;\cite{duchi2015asynchronous} (more precisely Assumption A). However, in that paper the delays are seen as random variables and bounded in $L^p$ and thus differ from the deterministic treatment we propose.}.}

\begin{proposition}[epoch scaling with delays]\label{prop:scaling}
For $M>1$ machines\modif{\footnote{For $M=1$ machine, we have $k_m=m$ as mentioned in Section~\ref{sec:epoch} and we recover exactly the convergence rates of the vanilla proximal gradient.}}, uniformly over time:
\begin{itemize}
    \item if the delays are uniformly bounded by $d$ over the workers, i.e.\;$d_i^k \le d$ for all\;$i$, then $d\ge M $ and the epoch sequence has complexity $k_m = \mathcal{O}(m M)$;
    \item if the average delay is bounded by $\overline{d}$, i.e.\;$1/M \sum_{i=1}^M d_i^k \le \overline{d}$, then $\overline{d}\ge (M-1)/2$ and the epoch sequence has complexity $k_m = \mathcal{O}(m M)$.
\end{itemize}
\end{proposition}

The proof of this proposition is basic and reported in Appendix~\ref{apx:delays}. The detailed results are summarized in the following table.

\renewcommand{\arraystretch}{1.5}

\smallskip

\noindent\resizebox{\columnwidth}{!}{
\begin{tabular}{l|c|c}
&  uniform bound & average bound   \\
\hline
Condition & $d_i^k \le d$ for all $i$  & $\frac{1}{M} \sum_{i=1}^M d_i^k \le \overline{d} $ \\
\hline
Unimprov.\;bound & $d = M  + \tau ; ~~ \tau \ge 0$ & $\overline{d} = \frac{M - 1}{2} + \tau;   ~~ \tau \ge 0$\\
\hline
1 Epoch  & $k_{m+1}-k_m \le 2d+1  $ & $k_{m+1}-k_m \le 2M(2\overline{d}-M +3) -3  $ \\
\hline
Epoch sequence & $k_m \le (2M + 2\tau  +1)m $ & $k_m  \le 4M(\tau + 1)m $
\end{tabular}}

\medskip
Bounding the average delay among the workers is an attractive assumption which is however much less common in the literature. The defined epoch sequence and associated analysis subsumes this kind of assumption.

In the case of uniformly bounded delays, the derived link between epoch and time sequence  enables us to compare our rates in the strongly convex case (Theorem~\ref{thm:strong}) with the ones obtained for PIAG \cite{aytekin2016analysis,vanli2016global,vanli2016stronger}. To simply the comparison, let us consider the case where all the workers share the same strong convexity and smoothness constants $\mu$ and $L$. The first thing to notice is that the admissible stepsize for PIAG depends on the delays' uniform upper bound $d$ which is practically concerning, while the usual proximal gradient stepsizes are used for the proposed DAve-RPG. Using the optimal stepsizes in each case, the convergence rates in terms of time $k$ are:
\begin{center}
\begin{tabular}{c|c|c}
 & DAve-RPG    &  PIAG \\
 \hline
Reference  & Th.~\ref{thm:strong}   &  Th.~3.4 of \cite{vanli2016stronger} \\
Stepsize  & $\gamma = \frac{2}{\mu+L}$   &  $\gamma = \frac{16}{\mu}\left[(1+\frac{\mu}{48 L} )^{\frac{1}{d+1}} -1 \right]$ \\
Rate & $\left( 1 - \frac{2}{1+\frac{L}{\mu}} \right)^{\frac{k}{d+0.5}} $   &  $\left( 1 - \frac{1}{49\frac{L}{\mu}} \right)^{\frac{k}{d+1}} $ \\
\end{tabular}    
\end{center}

\smallskip

We notice in both cases the exponent inversely proportional to the maximal delay $d$ but the term inside the parenthesis is a hundred times smaller for PIAG. Even if our algorithm is made for handling the flexible delays, this comparison illustrates the interest of our approach over PIAG for distributed asynchronous optimization in the case of bounded delays. 

\section{Numerical Illustrations}
\label{sec:experiments}

In this section, we run some numerical experiments to illustrate the behavior of our algorithm in the general convex case: we compare with the synchronous version and state-of-the-art method PIAG; we also point out the effect of repeated local iterations. 
\modifTwo{These experiments complement the ones of the companion short paper \cite{icml2018} which presents results for strongly convex function, different worker loads, and increasing number of machines.}

We consider the problem of minimizing the logistic loss with the $\ell_1$ and 
$\ell_2$-regularization on a dataset split among the workers. The problem reads
\begin{align*}
\min_{x\in\mathbb{R}^n} \frac{1}{M} \sum_{i=1}^M \underbrace{\sum_{j\in\mathcal{S}_i} \log\left(1 + \exp(-b_j a_j^\top x)\right) + \frac{\lambda_2}{2}\|x\|_2^2 }_{f_i(x)}  + \lambda_1 \|x\|_1,
\end{align*}
where for each example $j$, the pair $(a_j, b_j)$ represents the features $a_j\in\mathbb{R}^n$ together with the corresponding label $b_j\in\{-1,1\}$; and $\mathcal{S}_i$ represents the examples stored locally at machine $i$; the total number of examples is denoted by $m$.

The experiments were run on a CPU cluster, one core corresponding to one worker. Each core had 4 GB of memory and used one thread to produce updates. The code was written in Python using standard libraries only. The datasets used for the experiments are Criteo ($n=1,000,000$, $m=45,840,617$), URL ($n= 3,231,961$, $m= 2,396,130$), and KDDA ($n=20,216,830$, $m= 8,407,752$) from the LIBSVM datasets library \cite{chang2011libsvm}. 

\begin{figure*}[ht!]
\begin{tikzpicture}[scale=1]
\begin{groupplot}[group style={group name=plot2, group size= 2 by 2},width=\textwidth]

\nextgroupplot[ 
 width=0.45\columnwidth, 
 height=0.3\columnwidth, 
 xmin=0, 
xmax= 5000,
 xmajorgrids, 
ymode=log,
 yminorticks=true, 
 ymajorgrids, 
 yminorgrids, 
 ylabel={Suboptimality },
 xlabel={Wallclock time (s)},
 ]

 \addplot [  each nth point=3, filter discard warning=false, unbounded coords=discard, 
 color=blue, 
 solid, 
 line width=2.0pt, 
 mark size=2.5pt, 
 mark=none
 ] 
file {  num_res/criteo/dave_rpg.csv };
 \label{plots1:dave}

 \addplot [  each nth point=3, filter discard warning=false, unbounded coords=discard, 
 color=gray, 
 solid, 
 line width=2.0pt, 
 mark size=2.5pt, 
 mark=none,
 ] 
file {  num_res/criteo/synch_gd.csv };
 \label{plots1:gd}

 \addplot [  each nth point=3, filter discard warning=false, unbounded coords=discard, 
 color=green!40!black, 
 solid, 
 line width=2.0pt, 
 mark size=2.5pt, 
 mark=none
 ] 
file {  num_res/criteo/piag.csv };
 \label{plots1:piag}

\nextgroupplot[ 
 width=0.45\columnwidth, 
 height=0.3\columnwidth, 
xmin = 0,
xmax= 5000,
 xmajorgrids, 
 ymode=log,
 yminorticks=true, 
 ymajorgrids, 
 yminorgrids, 
 xlabel={Wallclock time (s)},
 ]

 \addplot [  each nth point=3, filter discard warning=false, unbounded coords=discard, 
 color=blue, 
 solid, 
 line width=2.0pt, 
 mark size=2.5pt, 
 mark=none
 ] 
file {  num_res/url/dave_rpg.csv };

 \addplot [  each nth point=3, filter discard warning=false, unbounded coords=discard, 
 color=gray, 
 solid, 
 line width=2.0pt, 
 mark size=2.5pt, 
 mark=none
 ] 
file {  num_res/url/synch_gd.csv };

 \addplot [  each nth point=3, filter discard warning=false, unbounded coords=discard, 
 color=green!50!black, 
 solid, 
 line width=2.0pt, 
 mark size=2.5pt, 
 mark=none
 ] 
file { num_res/url/piag.csv };

\end{groupplot}

\node[anchor=north] at ($(plot2 c1r1.south)+(0,-1.0)$) { \small (a) Criteo };

\node[anchor=north] at ($(plot2 c2r1.south)+(0,-1.0)$) { \small (b) URL };

\path (plot2 c1r1.north west|-current bounding box.north)--
      coordinate(legendpos)
      (plot2 c2r1.north east|-current bounding box.north);
\matrix[
    matrix of nodes,
    anchor=south,
    draw,
    inner sep=0.2em,
    draw
  ]at([yshift=1ex]legendpos)
  {
    \ref{plots1:dave}& \texttt{DAve-RPG} $p=1$ &[5pt]
\ref{plots1:gd}& Synchronous PG &[5pt]
\ref{plots1:piag}&   PIAG \\
};
\end{tikzpicture}
\caption{Performance on general convex functions ($\lambda_2=0$).}
\label{fig:comp}
\end{figure*}
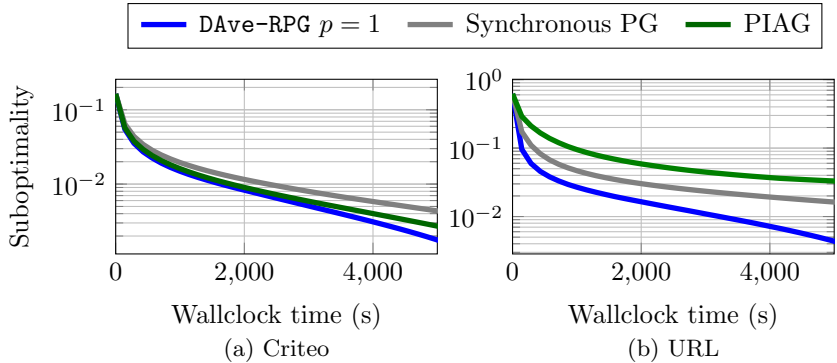

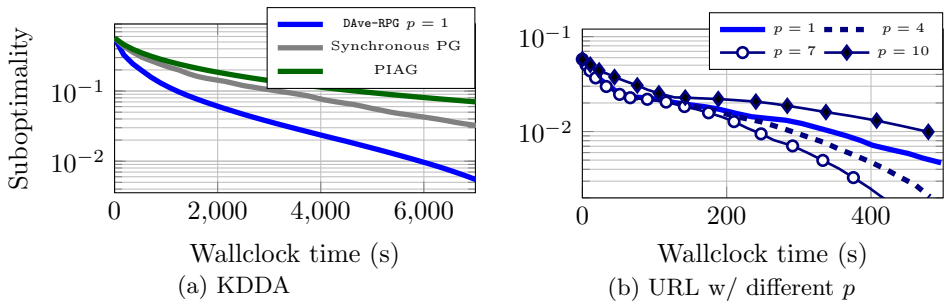
\begin{figure*}[ht!]
    \centering
    \begin{subfigure}[t]{0.49\textwidth}
        \centering
        \begin{tikzpicture}[scale=1]

\begin{axis}[ 
 width=1.0\columnwidth, 
 height=0.6\columnwidth, 
 xmin=0, 
 xmax = 7000,
 xmajorgrids, 
ymode=log,
 yminorticks=true, 
 ymajorgrids, 
 yminorgrids, 
 ylabel={Suboptimality },
 xlabel={Wallclock time (s)},
 legend style={font=\tiny},
 legend style={at={(1.0,1.1)}}
 ]

 \addplot [  each nth point=3, filter discard warning=false, unbounded coords=discard, 
 color=blue, 
 solid, 
 line width=2.0pt, 
 mark size=2.5pt, 
 mark=none
 ] 
file {  num_res/kdda/dave_rpg.csv };
\addlegendentry{\texttt{DAve-RPG} $p=1$}

 \addplot [  each nth point=3, filter discard warning=false, unbounded coords=discard, 
 color=gray, 
 solid, 
 line width=2.0pt, 
 mark size=2.5pt, 
 mark=none,
 ] 
file {  num_res/kdda/synch_gd.csv };
\addlegendentry{Synchronous PG}

 \addplot [  each nth point=3, filter discard warning=false, unbounded coords=discard, 
 color=green!40!black, 
 solid, 
 line width=2.0pt, 
 mark size=2.5pt, 
 mark=none
 ] 
file {  num_res/kdda/piag.csv };
\addlegendentry{PIAG}

\end{axis}

\end{tikzpicture}
        \vspace*{-0.7cm}
        \caption{KDDA \label{fig:strong_kdda}}
    \end{subfigure}%
    ~ 
    \begin{subfigure}[t]{0.49\textwidth}
        \centering
        \begin{tikzpicture}[scale=1]

\begin{axis}[ 
legend columns=2, 
 width=1.0\columnwidth, 
 height=0.6\columnwidth, 
 xmin=0, 
 xmax=500, 
 xmajorgrids, 
 ymode=log,
 ymin=2e-3, 
 ymax=0.12, 
 yminorticks=true, 
 ymajorgrids, 
 yminorgrids, 
 xlabel={Wallclock time (s)},
 legend style={font=\tiny},
 legend style={at={(1.0,1.1)}}
 ]

 \addplot [ 
 color=blue, 
 solid, 
 line width=2.0pt, 
 mark size=2.5pt, 
 mark=none
 ] 
file {  num_res/diff_p/asynch_gd_1.txt };
\addlegendentry{$p=1$}

 \addplot [  
 color=blue!50!black, 
 dashed, 
 line width=2.0pt, 
 mark size=2.5pt, 
 mark=none
 ] 
file {  num_res/diff_p/asynch_gd_4.txt };
\addlegendentry{$p=4$}

 \addplot [ 
 color=blue!50!black, 
 solid, 
 line width=1.0pt, 
 mark size=2.0pt, 
 mark=*,
mark options={fill=white},
mark repeat={2}
 ] 
file {  num_res/diff_p/asynch_gd_7.txt };
\addlegendentry{$p=7$}

 \addplot [  
 color=blue!50!black, 
 solid, 
 line width=1.0pt, 
 mark size=2.5pt, 
 mark=diamond*,
mark options={fill=black},
mark repeat={2}
 ] 
file {  num_res/diff_p/asynch_gd_10.txt };
\addlegendentry{$p=10$}

\end{axis}

\end{tikzpicture}
        \vspace*{-0.25cm}
        \caption{URL w/ different $p$ \label{fig:strong_p}}
    \end{subfigure}
    \vspace*{-0.5cm}
    \caption{Performance on strongly convex functions.}
\end{figure*}

In Fig.~\ref{fig:comp}, we plot the suboptimality versus wallclock time for the proposed DAve-RPG with $p=1$, the usual synchronous proximal gradient, and PIAG \cite{aytekin2016analysis}. For each of the datasets, we use the first 100,000 features, and split evenly the examples over 50 workers. We take $\lambda_1 = 10^{-11}$ and $10^{-7}$ respectively and $\lambda_2 =0$ for both. As we do not use $\ell_2$-regularization, the problem is not strongly convex and the rate is not linear. However, it is clear that, just as the synchronous proximal gradient descent, DAve-RPG appears to converge with rate $O(\frac{1}{k})$, in line with Theorem \ref{gradient_rates_theorem}. For all algorithms, we used the maximal stepsize (for PIAG, we took the limit $\mu\rightarrow 0$ in \cite{aytekin2016analysis}). 
Even in this case where the workers have similar computational loads, the performance of DAve-RPG is clearly better than that of the synchronous gradient descent. DAve-RPG also outperforms PIAG, notably thanks to its robustness (as expected from Fig.~\ref{fig:2d_examples}).

In Fig.~\ref{fig:strong_kdda}, we use a non-zero $\ell_2$-regularization, leading to a strongly convex problem: we plot the suboptimality versus wallclock time for the proposed DAve-RPG with $p=1$, the usual synchronous proximal gradient, and PIAG for the KDDA dataset. We use the first 200,000 features, and split evenly the examples over 60 workers. In this, the performance gain brought by DAve-RPG is even more significant.
Finally, in Fig.~\ref{fig:strong_p},  we illustrate the repetition of local iterations: we plot the suboptimality versus wallclock time for the proposed DAve-RPG with $p=1, 4, 7, 10$ on the full URL dataset with $\lambda_1 = 10^{-6}$ and $\lambda_2=1/m$ split evenly over 100 workers. We see that a tradeoff appears between computation and communications/updates; in this particular case, the performance improves up to $p=7$ and then degrades afterwards.

\section{Conclusions}
\label{sec:conclusions}

This paper describes a novel algorithm for asynchronous distributed optimization. 
A key property of this algorithm is that it
does not require unrealistic assumptions on machine delays. It is based on two original algorithmic features. First, the master machine keeps a combination of the output of all the workers last repeated proximal gradient steps, whereas for most algorithms in the literature, the master performs a step using the last gradients computed by the workers. Second, the workers can freely choose how many proximal gradient repetitions they make, leading to scarcer exchanges and more flexible communications. 

These special features lead us to two key theoretical findings: i) an epoch-based analysis adapted to any kind of delays; and ii) the use of the same stepsizes as in the classical proximal gradient algorithm.
We proved the convergence of the algorithm in the general case and with a linear rate in the strongly convex case. Although long delays may slow down the algorithm, it still converges both in theory and in experiments without being biased by more frequently updating workers.

The analysis suggests that some of the provided ideas may be used if updates are performed differently. Just in the way the vanilla proximal-gradient algorithm and its analysis form a base for studying advanced methods, we believe that the proposed algorithm and its original analysis may serve for future works in distributed optimization.

\section*{Acknowledgments}

We thank Robert Gower for valuable comments on the first versions of this paper.

\bibliographystyle{siamplain}
\bibliography{dist_optim}

\appendix

\section{Proof of convergence in the general case}
\label{apx:gen}
\modif{This appendix completes the proof of Theorem\,\ref{general_conv_thm} given in the main text to lift the unique minimizer assumption using an additional boundedness assumption on delays and inner loops.}

Let $X^\star$ be the set of minimizers of \eqref{eq:pb}, and fix $x^\star \in X^\star$. We are going to show the existence of another minimizer $x\in X$ having properties controlled with the two additional assumptions.

We use first the additional assumption that the number of inner loops is uniformly bounded by $p<\infty$. We define sequence $\mathbf{a'}^k$ by
\begin{align*}
    \mathbf{a'}^k \coloneqq \beta^{p-1}\left\|\overline x^{k} - \overline x^\star\right\|^2 + \left(1 - \beta^{p-1}\right)\max\left(\left\|\overline x^{k} - \overline x^\star\right\|^2, \left\|\overline x^{k}_{-i(k)} - \overline x^\star_{-i(k)}\right\|^2\right) 
\end{align*}
with $\beta \coloneqq \min_i \pi_i$. Following~\eqref{eq:recursion_for_limited_p_and_d} in the proof of Lemma \ref{lem:main}, we still have the bound
\begin{align}
\label{eq:renewed_c_k_bound}
    \left\|x_i^k - x_i^\star\right\|^2 \le \mathbf{a'}^{k - D_i^k}.
\end{align}
Furthermore, \eqref{eq:bound_by_max_j} and \eqref{eq:bound_by_max_j_2} imply that 
$\mathbf{a'}^k \le \max_i \mathbf{a'}^{k - D_i^k}$.

We use now the additional assumption that $(D_i^k)$ are bounded by $D$. We introduce 
\begin{align*}
    \mathbf{e}^k := \max_{0\le d < D} \mathbf{a'}^{k + d}.
\end{align*}
for a fixed $k>0$, we write $k=mD + r$ with $m=\lfloor k/D \rfloor$ and $r=k-mD$. 
We can prove by induction (as in the proof of Theorem~\ref{thm:strong}), that for any $r\in[0,D]$ that
\begin{align*}
    \mathbf{e}^m_r := \max_{0\le d < D} \mathbf{a'}^{mD + r + d} \le \mathbf{e}^{m-1}_r.
\end{align*}
thus, have $ \mathbf{e}^m_r \rightarrow \mathbf{b}_r$ for some $\mathbf{b}_r$. In addition, we have for any $r$ and $r'> r$ that $\mathbf{e}^m_{r'} \le \max(\mathbf{e}^m_r, \mathbf{e}^{m+1}_r)$ as the latter maximum covers the interval of the former (see Fig.~\ref{fig:max_over_interval}) thus $\mathbf{b}_{r'} \le \mathbf{b}_r $. Similarly, we have $\max(\mathbf{e}^{m-1}_{r'}, \mathbf{e}^{m}_{r'})  \ge  \mathbf{e}^m_r$ which gives the reverse inequality; thus $\mathbf{b}_{r'} = \mathbf{b}_r = \mathbf{b}$.
\begin{figure}[h!]
\centering
    \includegraphics[scale=0.25]{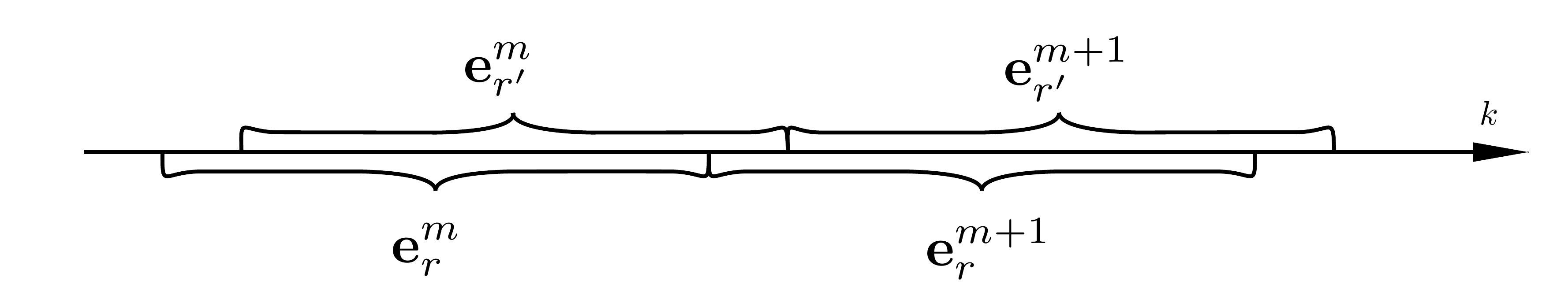}
\caption{Maxima over covering intervals of times}
\label{fig:max_over_interval}
\end{figure}

Thus we have that the sequence $(\mathbf{e}^k)$ is the union of $D$ sequences converging to $\mathbf{b}$ and thus converges itself to $b$.
Moreover, using \eqref{eq:renewed_c_k_bound}, we get that for any $i=1,..,M$
\begin{align*}
&\limsup_{k\rightarrow\infty}\left\| x_i^k - x_i^\star\right\|^2\le \mathbf{b}, \\
&\limsup_{k\rightarrow\infty}\left\|\overline x^k - \overline x^\star\right\|^2\le \mathbf{b}, ~~~\text{ and }~~~ \limsup_{k\rightarrow\infty} \left\|\overline x^{k}_{-i(k)} - \overline x^\star_{-i(k)}\right\|^2\le \mathbf{b}.
\end{align*}
This implies that the $\limsup$ of the second term in $\mathbf{a'}^k$ is upper bounded by $ \mathbf{b}$ and so is the maximum over $D$ consecutive times. Thus we have that for any $\varepsilon>0$, there is a $K$ such that for all $k>K$, 
\begin{align}
    & ~~~\mathbf{b}-\varepsilon \leq (1-\beta^p) \max_{0\le d< D}\left\|\overline x^{k+d} - \overline x^\star\right\|^2 + \beta^p( \mathbf{b}-\varepsilon) \nonumber \\
\text{thus }~~~    & ~~~\mathbf{b}- \frac{1+\beta^p}{1-\beta^p}\varepsilon \leq  \max_{0\le d< D}\left\|\overline x^{k+d} - \overline x^\star\right\|^2 \leq \mathbf{b} + \varepsilon \nonumber \\
\text{so }~~~   \label{eq:max_conv}
    & ~~~\max_{0\le d< D}\left\|\overline x^{k+d} - \overline x^\star\right\|^2\rightarrow \mathbf{b}.
\end{align}
This convergence yields in turn that $\|\overline x^{k} - \overline x^\star\|^2 \to \mathbf{b}$; for better readability, we postpone the proof of this fact at the end of this section. 

We have now all the ingredients to establish convergence of $(x^k)$ in the case of multiple minimizers. 
In the proof of Th.~\ref{general_conv_thm} for a unique minimizer (in Sec.~\ref{sec:cv_gen} of the main text), the uniqueness of the minimizer is used only that the last steps. All the previous arguments could be repeated here to 
establish the existence of a subsequence of $(\overline x^k)$ converging to $\overline x$ with $x = \prox_{\gamma g}(\overline x)$ being an optimal point. So let us pick this special optimal point, as $x^\star$ used in the above analysis. Since $\left\| \overline x^k -   \overline x^\star\right\|^2\to\mathbf{b}$, this limit can be only equal to $0$, which directly implies that $x^k\rightarrow x^\star$, and ends the proof.

\medskip

\noindent
\textbf{Proof of the statement that $\|\overline x^{k} - \overline x^\star\|^2 \to \mathbf{b}$.}   

\noindent
We will establish the convergence by contradiction. Let $(n_m)$ be a diverging sequence such that $\|\overline{x}^{n_m}-x^\star\|^2 \leq \mathbf{b}-\varepsilon$  for some $\varepsilon > 0$. From \eqref{eq:max_conv} we have that there also exists a sequence $(l_{m})$ such that $\left\|\overline x^{l_{m}}-x^\star\right\|^2\rightarrow \mathbf{b}$ and $l_{m+1} - l_{m} \le D' < \infty $. Thus, for any $\delta > 0$, there is $K<\infty$  such that for any $k>K, m>K$, 
$$
\left\|x_i^{k}-x_i^\star\right\|^2\le \mathbf{b} + \delta,\quad 
    \left\|\overline{x}^{l_{m}} - x^\star\right\|^2\ge \mathbf{b} - \delta,\quad\text{and}\quad
    \left\|\overline{x}^{n_m} - x^\star\right\|^2\le \mathbf{b} - \varepsilon.
$$
For any moment $n=n_m$ and $l=l_{m}$ fulfilling $l_{m-1}<n_m\le l_{m}$ and $m > K$, denote by $i$ the agent updating at time $n$. Let $u+1$ be the number of updates of $i$ between $n$ and $l$, and let $n=s_0<s_1<\dotsb< s_u\le l$ be the moments of these updates, we get for any $q=1,..,u$ that
\begin{align}
\label{eq:lim_A}
    \left\|\overline x^{s_q}- \overline x^\star\right\|^2&\le \sum_{j=1}^M\pi_j\left\|x_j^{s_q} - x_j^\star\right\|^2
    \le\sum_{j\neq i}^M\pi_j\left\|x_j^{s_q} - x_j^\star\right\|^2 + \pi_i \mathbf{a'}^{s_{q-1}}\nonumber\\
    &\le (1 - \pi_i)\left(\mathbf{b}+\delta\right) + \pi_i\beta^{p-1}\left\|\overline x^{s_{q-1}} - \overline x^\star\right\|^2 + \pi_i\left(1 - \beta^{p-1}\right)\left(\mathbf{b} + \delta\right)\nonumber\\
    &\le \left(1 - \phi \right)(\mathbf{b} + \delta) + \phi \left\|\overline x^{s_{q-1}} - \overline x^\star\right\|^2\nonumber
\end{align}
with $\phi := \pi_i \beta^{p-1}$. Thus, by induction for $q=1,..,u$,  
\begin{align*}
    \left\|\overline x^{s_u}-\overline x^\star\right\|^2&\le (1-\phi^u)(\mathbf{b} + \delta) + \phi^u  \left\|\overline x^{n} - \overline x^\star\right\|^2\\
    &\le (1-\phi^u)(\mathbf{b} + \delta) + \phi^u (\mathbf{b} - \varepsilon) = \mathbf{b} + \delta - \varepsilon \phi^u
\end{align*}
As $u < D'$, we obtain \\
\resizebox{\textwidth}{!}{$\displaystyle \mathbf{b}-\delta\le \left\|\overline x^l-x^\star\right\|^2\!\le \sum_{j\neq i}^M \pi_j\left\|x_j^{l} - x_j^\star\right\|^2\! + \pi_i \mathbf{a'}^{s_u} \le \mathbf{b} + \delta - \varepsilon\beta^{p(u + 1)}\le \mathbf{b} + \delta - \varepsilon \pi_i \phi^{D'}\!\!.$}\\
This yields $\delta \geq \varepsilon \pi_i \phi^{d'}/2>0$ which contradicts the arbitrariness of $\delta$, and then proves that $
    \left\|\overline{x}^k - \overline x^\star\right\|^2\rightarrow \mathbf{b}.$

\smallskip

\section{Proof of the rate of convergence}
\label{apx:rate}
This appendix presents the proof of Theorem \ref{gradient_rates_theorem}. We first introduce some notation and establish a key lemma. 

Pick any $x^\star$ in the set of minimizers of $F$. We are going to bound the maximal sum of three terms $g^k,\ o^k$ and $r^k$ defined as follows as means of quantities over all the machines. For technical reasons, we also need to define $g_{-i}^k,\ o_{-i}^{k}$ and $ r_{-i}^{k}$ for all $i$, as means of the same quantities without $i$-th summand. Specifically,
\begin{align*}
        r^k &\coloneqq\Sum\pi_i\left\|x_i^k - x_i^\star - (\overline x^k - \overline x^\star) \right\|^2,  & r_{-i}^k &\coloneqq \sum_{j\neq i}\pi_j\left\|x_j^k - x_j^\star - (\overline x^k - \overline x^\star) \right\|^2,\\
        g^k &\coloneqq \Sum\omega_i\pi_i \left\|\nabla f_i(z_i^k) - \nabla f_i(x^\star)\right\|^2,  & g_{-i}^k &\coloneqq \sum_{j\neq i}\omega_j\pi_j \left\|\nabla f_j(z_j^k) - \nabla f_j(x^\star)\right\|^2,\\
        o^k &\coloneqq \Sum\pi_i \left\|z_i^k - w_i^k - (x^\star - \overline x^\star)\right\|^2,  & o_{-i}^k &\coloneqq \sum_{j\neq i}\pi_j \left\|z_j^k - w_j^k - (x^\star - \overline x^\star)\right\|^2,
                    \end{align*}
     \begin{align*}   
        s^k &\coloneqq \min_{k' \le k}\left(r^{k'} + g^{k'} + o^{k'}\right), & s_{-i}^k &\coloneqq \min_{k' < k}\left(r_{-i}^{k'} + g_{-i}^{k'} + o_{-i}^{k'}\right),
    \end{align*}
where $z_i^k$ and $w_i^k$ satisfy $x_i^k = z_i^k - \gamma_i \nabla f_i(z_i^k)$ and $z_i^k = \prox_{\gamma g}(w_i^k)$; $\omega_i =  \gamma_i (  2/L_i  - \gamma_i )$. The quantity $s^k$ controls the decrease of the error in the algorithm, as formalized in Lemma~\ref{e_m_lemma}. The others quantities are involved in the following three useful inequalities. Using variance decomposition, we get that 
    \begin{align}
    \label{r_k_lemma}
        \left\|\overline x^k - \overline x^\star \right\|^2 &= \Sum\pi_i \left\|x_i^k - x_i^\star \right\|^2 - r^k,\\
        \left\|\overline x_{-i}^k - \overline x_{-i}^\star \right\|^2 &= (1 - \pi_i)^{-1}\sum_{j\neq i}\pi_j \left\|x_j^k - x_j^\star \right\|^2 - r_{-i}^k.\nonumber
    \end{align}
From the smoothness of the $(f_i)$, we have 
    \begin{align}
    \label{eq:g_k_bound}
        \Sum\pi_i\left\|x_i^k - x_i^\star\right\|^2&\le \Sum\pi_i\left\|z_i^k - x^\star\right\|^2 - g^k,\\
        (1 - \pi_i)^{-1}\sum_{j\neq i}\pi_j\left\|x_j^k - x_j^\star\right\|^2&\le (1 - \pi_i)^{-1}\sum_{j\neq i}\pi_j\left\|z_j^k - x^\star\right\|^2 - g_{-i}^k\nonumber.
    \end{align}
Finally, by~\eqref{eq:z_i}, we also have 
    \begin{align}
    \label{o_k_lemma}
        \Sum\pi_i\left\|z_i^k - x^\star\right\|^2 &\le \Sum\pi_i \mathbf{a}^{k - D_i^K} - o^k,\\
        (1 - \pi_i)^{-1}\sum_{j\neq i}\pi_j\left\|z_j^k - x^\star\right\|^2 &\le (1 - \pi_i)^{-1}\sum_{j\neq i}\pi_j \mathbf{a}^{k - D_j^k} - o_{-i}^k.\nonumber
    \end{align}

\begin{lemma}
\label{e_m_lemma}
For any $k\in[k_m, k_{m+1})$, we have
    \begin{align*}
        \mathbf{b}^m \le \mathbf{b}^{m-2} - s^k.
    \end{align*}
    with $ \mathbf{a}^k$ defined by \eqref{eq:ak} and $\mathbf{b}^m = \max_{k\in [k_m, k_{m+1})} \mathbf{a}^k$ as in the proof of Theorem~\ref{general_conv_thm}.
\end{lemma}

\begin{proof}
Combining Eqs.~\eqref{r_k_lemma}, \eqref{eq:g_k_bound}, and \eqref{o_k_lemma}, we get for any $k\in [k_m, k_{m+1})$
\begin{align}
       \label{eq:a_k_recursion}
       \left\|\overline x^k - \overline x^\star\right\|^2 &= \Sum\pi_i \left\|x_i^{k} - x_i^\star\right\|^2 - r^k \le \Sum\pi_i\left\|z_i^k - x^\star\right\|^2 - r^k - g^k \nonumber \\
        &\le \Sum\pi_i \mathbf{a}^{k - D_i^K} - r^k - g^k - o^k  \le \max_{k'\in [k_{m-1}, k)} \mathbf{a}^{k'}  - s^k \le \mathbf{b}^{m-1} - s^k
\end{align}
where the last inequality comes from two facts: (i) for any $k'\in [k_m, k_{m+1})$,  $\mathbf{a}^{k'}\le \mathbf{b}^m $ by definition and (ii) $\mathbf{b}^m \le \mathbf{b}^{m-1}$ (as shown in the proof of Theorem\;\ref{general_conv_thm}).
    
Similarly, if at moment $k$ the update is done by slave $i$, we have
\begin{align}
    \label{eq:b_k_recursion}
      \left\|\overline x_{-i}^k - \overline x_{-i}^\star\right\|^2     
        &\le (1 - \pi_i)^{-1}\sum_{j\neq i}\pi_j \mathbf{a}^{k - D_j^k} - g_{-i}^k - o_{-i}^k - r_{-i}^k \nonumber \\
        &\le \max_{k'\in[k_{m-1}, k_m):\ d_i^{k'}\neq 0} \mathbf{a}^{k'} - s_{-i}^k,
\end{align}
where $d_i^{k'}\neq 0$ comes from the fact that $k' = k - D_j^k$ was an update from a worker $j \neq i$ thus $d_i^k\neq 0$ (recall Fig.~\ref{fig:scheme}).
      
This can be wrapped up as
\begin{align}
    \label{eq:c_k_grad_recursion}
        \mathbf{a}^k  \le \max\left( \mathbf{b}^{m - 1} - s^k, \max_{k'\in[k_{m-1}, k_m):\ d_i^{k'}\neq 0} \mathbf{a}^{k'} - s_{-i}^k\right).
\end{align}
Denote by $j(k')$ the agent who is responsible for the update at moment $k'$. Then, plugging \eqref{eq:c_k_grad_recursion} into \eqref{eq:b_k_recursion} yields 
    \begin{align}
    \label{eq:b_k_grad_recursion}
         \left\|\overline x_{-i}^k - \overline x_{-i}^\star\right\|^2 \le & \max_{k'\in[k_{m-1}, k_m):\ d_i^{k'}\neq 0}\max \Big(  \mathbf{b}^{m-2} - s^{k'} - s_{-i}^k,\nonumber \\
         & ~~~ \max_{k''\in[k_{m-2}, k_{m-1}):\ d_{j(k')}^{k''}\neq 0} \left(\mathbf{a}^{k''} - s_{-j(k')}^{k'} - s_{-i}^k\right)\Big).
    \end{align}

By definition, $(s^k)$ and $(s_{-i}^k)$ are non-negative, non-increasing sequences; furthermore, for any $i$ and $j$ such that $i\neq j$ it holds that $s^k \le \max\left(s_{-i}^k, s_{-j}^k\right)$. Thus, \eqref{eq:b_k_grad_recursion} can be recast as
    \begin{align*}
       \left\|\overline x_{-i}^k - \overline x_{-i}^\star\right\|^2 \le \max\left(\mathbf{b}^{m-2} - s^k, \max_{k''\in[k_{m-2}, k_{m-1}):\ d_{j(k')}^{k''}\neq 0} \mathbf{a}^{k''} - s^k\right) \leq \mathbf{b}^{m-2} - s^k
    \end{align*}
    and finally, since $\mathbf{b}^{m-1}\le \mathbf{b}^{m-2}$, we obtain
    \begin{align*}
        \mathbf{b}^m = \max_{k\in [k_m, k_{m+1})} \mathbf{a}^k \le \max(\mathbf{b}^{m-1} - s^k, \mathbf{b}^{m-2} - s^k) = \mathbf{b}_{m-2} - s^k. 
    \end{align*}
\end{proof}

We are now in position to give the proof of  Theorem~\ref{gradient_rates_theorem}, establishing the rate of convergence of our algorithm.

\begin{proof}\textit{(of Theorem~\ref{gradient_rates_theorem})}
Applying $m/2$ times Lemma \ref{e_m_lemma} and using that $(s^k)$ is non-increasing, we get
\begin{align*}
\mathbf{b}^m \le \mathbf{b}^0 - \frac{m}{2}s^k,
\end{align*}
We deduce       
\begin{align*}       
 s^k \le \frac{2(\mathbf{b}^0 - \mathbf{b}^m)}{m} 
 \leq 2\frac{\max_{k'\in[k_0, k_1)}\mathbf{a}^{k'}}{m}.
\end{align*}
Using that $  \|\overline x^k - \overline x^\star \|^2 \le \max_i\left\|x_i^k - x_i^\star\right\|^2$ and $  \|\overline x_{-i}^k - \overline x_{-i}^\star \|^2 \le \max_{i}\left\|x_i^k - x_i^\star\right\|^2$, we deduce from Lemma \ref{lem:main} that
    \begin{align*}
       & \mathbf{a}^k
        \le \max_i\left\|x_i^k - x_i^\star\right\|^2 
        \le \max_{k'\le k-1} \mathbf{a}^{k'}
        \le\dotsb
        \le \mathbf{a}^0
        \le \max_i\left\|x_i^0 - x_i^\star\right\|^2 \\
        \text{and } &    s^k   \le \frac{2\max_i \left\|x_i^0 - x_i^\star\right\|^2}{m}.
    \end{align*}

On the other hand, we have that $x^k = \prox_{\gamma g}(\overline x^k)$ satisfies $\overline x^k - x^k \in \gamma \partial g(x^k)$ (see e.g.\;\cite[Prop.~16.34]{bauschke2011convex})
We then introduce
    \begin{align*}
        h^k \coloneqq (\overline x^k - x^k)/\gamma + \nabla f(x^k)\in \partial F(x^k).
    \end{align*}

Writing $\overline x^k$ as $\sum_i\pi_i (z_i^k - \gamma_i \nabla f_i(z_i^k))$ and using each $f_i$'s smoothness, we have
\begin{align}\label{eq:boundh}
        \|h^k\|^2 
        &= \left\|\gamma^{-1}\Sum\pi_i \left( z_i^k - x^k \right) -  \Sum\pi_i \left(\nabla f_i(z_i^k) - \nabla f_i(x^k)\right)\right\|^2 \nonumber\\
        &\le \frac{1}{\gamma^2}\Sum\pi_i \left\| \left(z_i^k - x^k \right) -\gamma_i \left(\nabla f_i(z_i^k) - \nabla f_i(x^k)\right)\right\|^2 
        \nonumber\\
        &\le \frac{1}{\gamma^2}\Sum\pi_i \left( \left\| z_i^k - x^k \right\|^2 - \omega_i \left\|\nabla f_i(z_i^k) - \nabla f_i(x^k)\right\|^2 \right) \nonumber\\
        &\le \frac{1}{\gamma^2 \sum_i \gamma_i^{-1}}\Sum \gamma_i^{-1} \left\| z_i^k - x^k \right\|^2.
    \end{align}
    Then, as $        \|a+b+c\|^2 \le \left(1 + \frac{\delta}{2}\right)\|a+b\|^2 + \left(1 + \left(\frac{\delta}{2} \right)^{-1}\right) \|c\|^2 \le \left( 2 + \delta \right) \left( \|a\|^2 + \|b\|^2 \right)$ $  + \left( 1+2\delta^{-1} \right)\|c\|^2
$
   for any $\delta > 0$, we can bound each summand with individual $\delta_i$:
    \begin{align*}
        &\|x^k - z_i^k\|^2
        = \left\|\prox_{\gamma g}(\overline x^k) - \prox_{\gamma g}(w_i^k)\right\|^2
        \le \|\overline x^k - w_i^k\|^2 \\
        &= \left\|\left(\overline x^k - \overline x^\star + x_i^\star - x_i^k\right) - \left(w_i^k - \overline x^\star + x_i^\star - x_i^k\right) \right\|^2 \\
        &= \left\|-\left(x_i^k - x_i^\star - (\overline x^k - \overline x^\star)\right) + \left(z_i^k - w_i^k - (x^\star - \overline x^\star)\right) - \gamma_i\left(\nabla f_i(z_i^k) - \nabla f_i(x^\star)\right)\right\|^2 \\
        &\le \left( 2 + \delta_i \right)\left(\left\|x_i^k - x_i^\star - (\overline x^k - \overline x^\star)\right\|^2 + \left\|z_i^k - w_i^k - (x^\star - \overline x^\star)\right\|^2 \right)\\
        &~~~ + \left(1 + 2\delta_i^{-1} \right) \gamma_i^2\left\|\nabla f_i(z_i^k) - \nabla f_i(x^\star)\right\|^2 .
    \end{align*}

    \begin{align*}
    \textrm{Hence, ~~~}      \Sum\pi_i \|x^k - z_i^k\|^2 \le \max_i\left[\left( 2 + \delta_i \right)\left( r^k + o^k \right) + \left(1 + 2\delta_i^{-1} \right)\frac{\gamma_i^2}{\omega_i}g^k \right]. 
    \end{align*}
    Thus \eqref{eq:boundh} gives
    \begin{align*}
        \min_{k'\le k}\|h^{k'}\|^2
        &\le \gamma^{-2}\min_{k'\le k}\max_i \left( \left( 2 + \delta_i \right) \left( r^{k'} + o^{k'} \right) + \left( 1 + 2\delta_i^{-1} \right)\frac{\gamma^2}{\omega}g^{k'} \right)\\
        &\le \gamma^{-2}\max_i\max\left(2 + \delta_i , \gamma_i^2\omega_i^{-1}\left( 1 + 2\delta_i^{-1} \right) \right) s^k.
    \end{align*}
    Taking $\delta_i = \frac{\gamma_i}{2/L_i - \gamma_i}$ yields
    \begin{align*}
        \min_{k'\le k}\|h^{k'}\|^2
        &\le \gamma^{-2}\max_i\left(2 + \delta_i \right) s^k
        \le \gamma^{-2}\max_i \left( 2 + \frac{\gamma_i}{2/L_i - \gamma_i} \right)\frac{2\max_j \left\|x_i^0 - x_i^\star\right\|^2}{m}\\
        &\le \frac{8\max_j \left\|x_j^0 - x_j^\star\right\|^2}{m \gamma^2 \min_i \left(2 - \gamma_i L_i \right)}.
    \end{align*}
    where at the last step we used our assumption $\gamma_i\in(0, 2/L_i)$.
\end{proof}

\section{Proof of epoch scaling with delays}
\label{apx:delays}
This appendix gives the proof of the results of Proposition\;\ref{prop:scaling} and the following table. 

\medskip
\noindent
\emph{Case of delays uniformly bounded by $d$.} By definition of time, we have $d\geq M$, and then $d=M+\tau$ with $\tau\geq 0$. It is easy to see on the definition of the epoch sequence of Section\,\ref{sec:epoch} that $k_{m+1} - k_m \leq 2d+1$ as $d_i^{k_{m+1}}\leq d$ for all $i$. Then there was a least one update of each machine in $[k_{m+1}-d;k_{k+1}]$. Repeating this reasoning at  $k_{m+1}-d-1$, one gets that two update occured in $[k_{m+1}-2d-1, k_{m+1}]$ hence the result.

\medskip
\noindent
\emph{Case of average delay bounded by $\overline{d}$}. To prove that  $\overline{d} = \frac{M - 1}{2} + \tau$ with $\tau\ge 0$, one can notice that at any time $k$ there can be only one worker with a zero delay (the updating one), only one with a delay equal to 1, and so on. Consequently, the sum of the delays is at least $M(M - 1) / 2$ thus the average is at least $(M - 1) / 2$. 

We now look carefully at the epoch sequence. To simplify notation, we introduce $N := k_{m+1} - k_m$ and $i:=i(k_{m+1})$ the machine updating at moment $k_{m+1}$. We will consider two subcases depending on which worker performed the update at~$k_m$:

\begin{itemize}
    \item \textit{When $i(k_m)=i$}. In this case, there cannot be any other update of $i$ between $k_m$ and $k_{m+1}$. Indeed, by definition of $k_{m+1}$ it is the first moment when every machine has been updated at least twice since moment $k_m$, so for $i$ it has to be the second time (including $k_m$). Therefore,
    \begin{align*}
        \sum_{k=k_m}^{k_{m+1}-1} d_i^k = 0 + 1 + \dotsb + (N - 1) = \frac{N(N-1)}{2}.
    \end{align*}
      
\item \textit{When $i(k_m)\neq i$}. In this case, there is a moment $\widetilde k \in (k_m, k_{m+1})$ such that $i(\widetilde k) = i$. Since $d_i^{k_m} \ge 1$ and for any two numbers $a, b$ we have $a^2 + b^2 \ge \frac{(a+b)^2}{2}$,
    \begin{align*}
        \sum_{k=k_m}^{k_{m+1}-1} d_i^k 
        &= 1 + \dotsb + (\widetilde k - k_m) + 0 + 1 + \dotsb + (k_{m+1} - \widetilde k - 1) \\
        &= \frac{(\widetilde k - k_m)(\widetilde k - k_m + 1)}{2} + \frac{(k_{m+1} - \widetilde k - 1)(k_{m+1} - \widetilde k)}{2}\\
        &= \frac{1}{2}\left(\left(\widetilde k - k_m + \frac{1}{2}\right)^2 + \left(k_{m+1} - \widetilde k - \frac{1}{2} \right)^2 - \frac{1}{2} \right) \\
        &\ge \frac{(k_{m+1} - k_m)^2 - 1}{4} = \frac{N^2 - 1}{4} \geq  \frac{N(N-1)}{4}.
    \end{align*}
\end{itemize}
In both cases, we have
    \begin{align}
        \label{eq:p1}
         \sum_{k=k_m}^{k_{m+1}-1} d_i^k \geq \frac{N(N-1)}{4}.
    \end{align}
In addition, for any moment $k_m + l$ ($l\ge 0$); among workers $j\neq i$, at least $M - 2$ have a delay greater than 0, at least $M - 3$ have a delay greater than 1, etc. 
    \begin{align}
         \label{eq:p2}
        \sum_{j\neq i} d_j^{k_m+l} \ge \sum_{q=0}^{M-2}q = \frac{(M - 1)(M - 2)}{2}.
    \end{align}
Summing \eqref{eq:p1} and \eqref{eq:p2} over $l=0,\dotsc, N - 1$ we obtain
    \begin{align*}
      \sum_{k=k_m}^{k_{m+1} - 1}\sum_{j=1}^M d_j^k \ge N \frac{(M - 1)(M - 2)}{2} + \frac{N(N-1)}{4}.
    \end{align*}
Combining this with the fact that $\sum_{k=k_m}^{k_{m+1} - 1}\sum_{j=1}^M d_j^k \le MN\overline{d}$ for the average bound leads to the result.

\end{document}